\documentclass[11pt,oneside]{amsart}

\usepackage{amsmath,amssymb,amsthm,graphicx,color,bbm}
\usepackage[left=1in,right=1in,top=1in,bottom=1in,a4paper]{geometry}
\usepackage{cite}
\usepackage[british]{babel}
\usepackage{hyperref} 
\usepackage{enumitem}
\usepackage{mathtools}
\usepackage{comment}
\usepackage{subcaption}
\usepackage{lipsum}    
\hypersetup{
    colorlinks=true,
    linkcolor=blue,
    filecolor=magenta,      
    urlcolor=cyan,
    pdftitle={Overleaf Example},
    pdfpagemode=FullScreen,
    }

\usepackage{tikz}
\theoremstyle{plain}
\newtheorem{thm}{Theorem}[section]
\newtheorem{lem}[thm]{Lemma}
\newtheorem{prop}[thm]{Proposition}
\newtheorem{cor}[thm]{Corollary}

\newtheorem*{assumption*}{Assumption}
\theoremstyle{remark}
\newtheorem{rem}[thm]{Remark}

\newcommand{\ind}{{\mathbbm{1}}}
\newcommand{\1}[1]{{\ind\mkern -1.5mu}{\{#1\}}}

\DeclareMathOperator{\Tr}{Tr}
\DeclareMathOperator{\Hessian}{Hess}

\DeclareMathOperator{\E}{\mathbb E}
\renewcommand{\P}{\mathbb P}

\renewcommand{\epsilon}{\varepsilon}

\newcommand{\TV}{\mathrm{TV}}
\newcommand{\tra}{{\scalebox{0.6}{$\top$}}}

\newcommand{\eps}{\varepsilon}
\newcommand{\re}{{\mathrm{e}}}

\newcommand{\ud}{{\mathrm d}}

\newcommand{\R}{{\mathbb R}}

\newcommand{\N}{{\mathbb N}}

\newcommand{\RP}{{\mathbb R}_+}

\newcommand{\cA}{{\mathcal A}}
\newcommand{\cB}{{\mathcal B}}

\newcommand{\cD}{{\mathcal D}}

\newcommand{\cM}{{\mathcal M}}

\newcommand{\Id}{\mathrm{I}_d}

 \newcommand{\aref}[1]{{\textup{\ref{#1}}}}

\makeatletter
\def\namedlabel#1#2{\begingroup  
    #2%
    \def\@currentlabel{#2}%
    \phantomsection\label{#1}\endgroup
}
\makeatother

\newlist{myenumi}{enumerate}{10}
\setlist[myenumi]{leftmargin=0pt, labelindent=\parindent, listparindent=\parindent, labelwidth=0pt, itemindent=!, itemsep=1pt, parsep=4pt}

\newlist{thmenumi}{enumerate}{10}
\setlist[thmenumi]{leftmargin=0pt, labelindent=\parindent, listparindent=\parindent, labelwidth=0pt, itemindent=!}

\title[Non-asymptotic bounds for denoising diffusions]{Non-asymptotic bounds for forward processes in denoising diffusions: Ornstein-Uhlenbeck is hard to beat}
\author{Miha Bre\v sar$^1$ \and Aleksandar Mijatovi\'c$^{1,2}$}

\thanks{$^1$\textsc{Department of Statistics, University of Warwick, \& $^2$The Alan Turing Institute, London, UK}\\
Email: \texttt{miha.bresar.2@warwick.ac.uk} \& \texttt{a.mijatovic@warwick.ac.uk}\\
We thank Jose Blanchet, Alain Durmus and Éric Moulines for useful conversations during the work on the paper}


\begin{document}

\begin{abstract}
Denoising diffusion probabilistic models (DDPMs) represent a recent advance in generative modelling that has delivered state-of-the-art results across many domains of applications. Despite their success, a rigorous theoretical understanding of the error within DDPMs, particularly the non-asymptotic bounds 
required for the comparison of their efficiency, remain scarce.  
Making minimal assumptions on the initial data distribution, allowing for example the manifold hypothesis, this paper presents explicit non-asymptotic bounds on the forward diffusion error in total variation (TV), expressed as a function of  the terminal time
$T$. 

We parametrise  multi-modal data distributions in terms of the distance $R$ to their furthest modes and consider forward diffusions with additive and multiplicative noise. Our analysis rigorously proves that, under mild assumptions, the canonical choice of the Ornstein-Uhlenbeck (OU) process cannot be significantly improved in terms of reducing the terminal time $T$ as a function of $R$ and  error tolerance $\eps>0$. Motivated by data distributions arising in generative modelling, we also establish a cut-off like phenomenon (as $R\to\infty$) for the convergence to its invariant measure in TV of an OU process,
initialized at a multi-modal distribution with maximal mode distance $R$.
\end{abstract}

\maketitle

\section{Introduction}
Denoising diffusion probabilistic models (DDPMs) represent a recent advancement in machine learning that has delivered state-of-the-art results across various domains~\cite{sohl2015deep,song2020score,song2019generative,ho2020denoising}. These models take data samples, corrupt them by adding noise and learn to reverse this procedure. Executing this reverse process allows users to generate new samples from the distribution of the data. For data distributions in $\R^d$, the noising process is often selected to be a solution of a stochastic differential equation (SDE) with its invariant measure as the target noise distribution. Consequently, the  convergence of the algorithms depends on the converge of the underlying noising diffusion towards its invariant measure. Commonly used dynamics include Langevin and, in particular, Ornstein-Uhlenbeck (OU) processes. A natural question is whether adopting a different SDE, possibly with multiplicative noise, could significantly improve the convergence rate. Analysis of this question is the central theme of this paper.

\subsection{Sources of error in DDPMs}
\label{subsec:diff}

We briefly introduce DDPMs following~\cite{song2020score} and discuss sources of errors described in~\cite{chen2022sampling,conforti2023score,benton2024nearly}.
Let $d\in\N$ and $T\in[0,\infty)$.
A DDPM is based on an ergodic forward process $Y=(Y_t)_{t\in[0,T]}$ in $\R^d$ with an invariant distribution $\pi$. The process is initialised at $Y_0\sim\rho_0$, where $\rho_0$ denotes the data distribution, and follows the SDE dynamics
\begin{equation}
\label{eq:SDE}
    \ud Y_t = b(Y_t)\ud t + \sigma(Y_t)\ud B_t,\quad \text{with $Y_0\sim \rho_0$,}
\end{equation}
where $B$ is a $d$-dimensional Brownian motion, $b:\R^d\to \R^d$ and $\sigma:\R^d\to\R^{d\times d}$. In practice, for example, the law $\rho_0$ could be a distribution of all images with certain content and the aim of the DDPMs is to generate new images of the same type. The dimensionality of this problem is typically high, with $d$ being proportional to the number of pixels in the image.

With this setup, the time reversed process $(Y_{T-t})_{t\in[0,T]}$ is again a diffusion. Let $q_t$ denote the marginal density of the forward process $Y_t$ at time $t\in[0,T]$ and set $a \coloneqq \sigma\sigma^\intercal$, where $\sigma^{\intercal}$ denotes the transpose of the matrix $\sigma$. Then, the reverse process $(\overleftarrow{Y_t})_{t\in[0,T]}$, satisfying the SDE
\begin{equation}
\label{eq:reverse_SDE}
\ud \overleftarrow{Y_t} = -(b(\overleftarrow{Y_t})-\frac{1}{2}\nabla a(\overleftarrow{Y_t}) -\frac{1}{2}a(\overleftarrow{Y_t})\nabla \log q_{T-t}(\overleftarrow{Y_t}))\ud t + \sigma(\overleftarrow{Y_t})\ud B_t', \quad \overleftarrow{Y_0}\sim q_T,
\end{equation}
 has the same law as $(Y_{T-t})_{t\in[0,T]}\stackrel{d}{=}(\overleftarrow{Y_t})_{t\in[0,T]}$
(as usual, $\nabla$ denotes the gradient of a scalar function and $B'$ is a $d$-dimensional standard Brownian motion, see~\cite{Anderson,Cattiaux} for more details on reversed diffusions). 
If one could sample $\overleftarrow{Y_0}\sim q_T$, running the reverse diffusion $\overleftarrow{Y}$ to time $T$ would produce a sample from the data distribution $\rho_0$.  Initialising $\overleftarrow{Y}$ by sampling $\overleftarrow{Y_0}$ from the invariant probability measure $\pi$
of $Y$ and running $\overleftarrow{Y}$ to time $T$
leads to an approximate simulation algorithm for the data distribution  $\rho_0$.

Denote by $\mathcal{L}_{\mathrm{DDPM}}(\overleftarrow{Y}_T)$ the distribution of the output of the DDPM, initialised by $\overleftarrow{Y}_0\sim\pi$, with time parameter $T\in[0,\infty)$. 
The error of the DDPM, measured by the total variation distance from the data distribution $\rho_0$, consists of three different components~\cite[Thm~2]{chen2022sampling}:
\begin{equation}
\label{eq:stable_diff_error}
\|\rho_0-\mathcal{L}_{\mathrm{DDPM}}(\overleftarrow{Y}_T)\|_{\TV} \leq \|\P_{\rho_0}(Y_T\in\cdot)-\pi\|_{\TV} + \eps_{\mathrm{disc}}\sqrt{T} +\eps_{\mathrm{score}}\sqrt{T},
\end{equation}
where  $\eps_{\mathrm{disc}}$ and $\eps_{\mathrm{score}}$ correspond to the discretisation and score matching errors, respectively, while $\P_{\rho_0}$ denotes the law of the process following the SDE in~\eqref{eq:SDE}. We now briefly describe each of these sources of error.

First, as we do not have access to the initial distribution of the reverse process (i.e. the law $\P_{\rho_0}(Y_T\in\cdot)$ of the forward process $Y$ at time horizon $T$), we initialize the reverse process at the invariant measure $\pi$ of the forward process ($\pi$ is the standard Gaussian when $Y$ is the OU process). The error associated with the forward process decreases  in the time parameter $T$~\cite[Thm~1]{benton2024nearly}. This decay is well-known to be exponential in $T$ for the OU process and will be such for any exponentially ergodic diffusion $Y$, satisfying SDE~\eqref{eq:SDE}. However, while it is clear from practical applications that $T$ increases with $\rho_0$ based on data sets of images of large dimension, the dependence of  $T$ on the distant modes of the initial data distribution $\rho_0$ are theoretically not well understood.

Second, we do not have direct access to the coefficient $\log q_{T-t}$ in SDE~\eqref{eq:reverse_SDE} as it depends on the initial data distribution $\rho_0$. This quantity is learnt by minimising the score in (very) deep hierarchical models, leading to   the score matching error $\eps_\mathrm{score}$. Important for our results is the fact that the total error in the score matching step in DDPMs increases with $T$:  by~\eqref{eq:stable_diff_error} it is bounded by the square root $\sqrt{T}$ of the  time horizon of the forward diffusion $Y$, multiplied by  $\eps_\mathrm{score}$, see~\cite[Thm~1]{conforti2023score} for more details.  

Third, since both the forward and reverse processes are in continuous-time and the simulation requires discrete time steps, the discretisation error $\eps_{\mathrm{disc}}$ constitutes the final component of the error term. Akin to the score matching error, the discretisation error increases with the number of steps the algorithm takes, which in turn increases with the square root of the running time $\sqrt{T}$ of the forward and reverse SDEs~\cite{benton2024nearly,conforti2023score,chen2022sampling}.

Inequality in~\eqref{eq:stable_diff_error} demonstrate that it is essential to select a forward process with fast convergence towards the invariant measure, as choosing a smaller time parameter $T$ can reduce both discretisation and score-matching errors. 
In practice, most commonly used forward diffusion in DDPMs is  the OU process $X=(X_t)_{t\in[0,T]}$, i.e the solution of the SDE
\begin{equation}
\label{eq:OU}
\ud X_t = -\mu X_t + \sqrt{2}\ud B_t, 
\end{equation}
where the free drift parameter $\mu\in(0,\infty)$ is typically set to be  $\mu = 1$.
This choice of a forward process is convenient for a number of reasons, including its exponential ergodicity with Gaussian invariant measure $\pi_X$ of mean zero and covariance $\Id/\mu$ and the analytical tractability of its transition densities, making it a canonical choice in practical applications, see~\cite{song2019generative,song2020score,benton2024nearly} and the references therein.

The primary contribution of this paper is to prove rigorously that an ergodic diffusion following SDE~\eqref{eq:SDE}, when initialized with a multi-modal data distribution $\rho_0$  commonly encountered in practical applications~\cite{Fefferman16,song2019generative}, requires at least a time 
$T_c \approx \log R$ to reach stationarity. Here, $R$ represents the distance to the furthest mode in the initial data distribution $\rho_0$. This result indicates that the convergence time of the Ornstein-Uhlenbeck process, which we prove exhibits cut-off type behavior at $T_c\approx \log R$ (as $R\to\infty$),  cannot be substantially reduced by adopting a different diffusion model following SDE~\eqref{eq:SDE}, even if this model incorporates multiplicative noise. This contrasts with diffusions used in Markov Chain Monte Carlo (MCMC) methods, where it has been established that increasing the variance with multiplicative noise can significantly enhance the convergence rate (see~\cite{Douc,BM2023,MR2134115} and Appendix~\ref{app:sec:A} below  for a discussion of this phenomenon in the context of tempered Langevin diffusions).

The remainder of the paper is organised as follows.
Subsections~\ref{subsec:results},~\ref{subsec:lan}
and~\ref{subsec:discussion} describe our assumption on the initial data distribution and discuss our results in the context of tempered Langevin diffusions.
A short \href{https://youtu.be/hQvfpwI0UPk?si=A9L8YF9SsVAem0cE}{YouTube} presentation covering these topics is in~\cite{Presentation_MB}.
Section~\ref{sec:main} describes our general framework (beyond tempered Langevin diffusion) and states our main results (Theorem~\ref{thm:error} and Corollary~\ref{cor:main}), while Section~\ref{sec:proofs} provides the proofs of all our results.
See~\cite{Presentation_MB} for the second  \href{https://youtu.be/xjzVPOEkl44?si=Y9D0aLjvWXmLaSJA}{YouTube} presentation discussing our main general results, Theorem~\ref{thm:error} and Corollary~\ref{cor:main}, as well as their proofs.
Section~\ref{sec:conclusion} concludes the paper.

\subsection{Data distribution, the OU process and cut-off}
\label{subsec:results}
In this section we  formalise  the assumptions on the initial data distribution $\rho_0$ and state the result on the cut-off type phenomenon in the convergence of the OU process 
as the furthest mode of $\rho_0$ tends to infinity.
Denote by $\cM_1(\R^d)$ the family of probability measures on $\R^d$. 
For $Q_1,Q_2\in\cM_1(\R^d)$, the \textit{total variation distance} is given by $\|Q_1-Q_2\|_{\TV} \coloneqq \sup_{A\in\cB(\R^d)}|Q_1(A)-Q_2(A)|$, where $\cB(\R^d)$ is the Borel $\sigma$-algebra on $\R^d$.
Throughout we denote by $|\cdot|$ and $\langle\cdot,\cdot\rangle$ the Euclidean norm and the standard scalar product on $\R^d$, respectively. For any $x\in\R^d$ and $r\in[0,\infty)$, let $B(x,r):=\{y\in\R^d:|y-x|\leq r\}$. 

Let $\eps>0$ be a given error tolerance and consider multi-modal data distributions $\rho_0$ in $\cM_{1,R,\eps}(\R^d)$, appearing in practical applications~\cite[Sec.~1]{chen2022sampling}, parameterised by the distance $R$ of the furthest mode $x_0$ of $\rho_0$ from the origin.

\noindent \textbf{Assumption} \namedlabel{ass:mode}{{\color{purple}\texttt{(DATA)}}} \textit{Let $d\in\N$, $1>>\eps>0$ and $R>2$. The probability measure $\rho_0\in \cM_1(\R^d)$ is in $\cM_{1,R,\eps}(\R^d)$, if 
there exist mode $x_0\in\R^d\setminus\{0\}$ and small $0<\delta<<1$
such that $|x_0|=R(1+\delta)$, $b_\rho\coloneqq \rho_0(B(x_0,\delta R))>3\eps$ and 
$\rho_0(\R^d\setminus B(0,R(1+2\delta))<\eps/2$.}

\begin{figure}[hbt]
\centering
    \includegraphics[width=80mm]{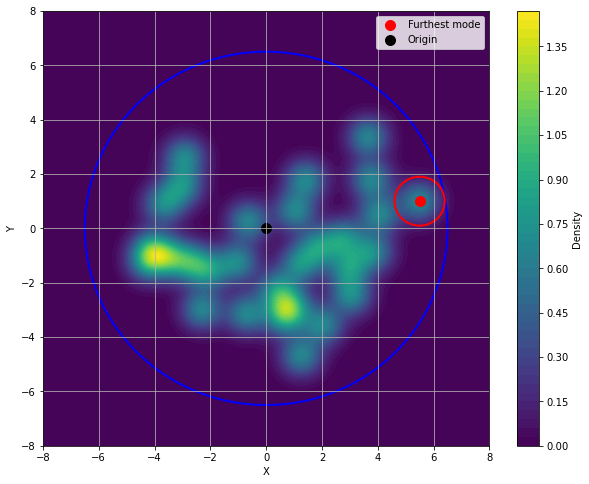}
    \caption{A heat map of a  two-dimensional multi-modal distribution $\rho_0$. The {\color{red}small} (resp. {\color{blue}large}) circle has radius {\color{red}$\delta R$}  (resp.  {\color{blue}$R(1+2\delta)$}). The mode {\color{red}$x_0$} of $\rho_0$ is the one furthest away from the origin with mass  $b_\rho$ typically orders of magnitude greater than $\eps$, i.e. $b_\rho>>\eps$.}
    \label{fig:A1distribution}
\end{figure}

\begin{rem}
\label{rem:mode}
 In~\ref{ass:mode} we make no assumptions on the \textit{moments}, \textit{existence of density} of $\rho_0$ or the \textit{KL-divergence} $\mathrm{KL}(\rho_0||\mathrm{Gauss})<\infty$, often assumed in the literature on DDPMs~\cite{benton2024nearly,chen2022sampling,conforti2023score}.
In particular,~\ref{ass:mode} allows for the \textit{manifold hypothesis}, asserting that the data distribution is supported on a submanifold of $\R^d$ of positive (possibly large) co-dimension~\cite{Fefferman16,de2022convergence}. 
\end{rem}
The literature has established that denoising diffusion models empirically outperform other methods in sampling from multi-modal distributions~\cite{song2019generative}. Moreover, multi-modality has been confirmed for many datasets used in practical applications~\cite{janati2024divide}.  This suggests the simple idea to parameterise the problem in terms of the mode furthest from the origin of the initial data distribution~$\rho_0$ in Assumption~\ref{ass:mode}.   To the best of our knowledge this parametrisation of the problem is novel and, moreover, crucial for  the development of the framework and the results in the paper.

The following proposition establishes the cut-off type phenomenon for the convergence of the OU process initialised at a multi-modal distribution in $\cM_{1,R,\eps}(\R^d)$ with large distance  between the modes. Its role is to motivate our  main results in Theorem~\ref{thm:error} and Corollary~\ref{cor:main} for general forward processes. The proof of Proposition~\ref{cor:cutoff}, relying on explicit transition densities of the OU process, is given in Section~\ref{sec:proofs} below.

\begin{prop}\label{cor:cutoff}
Let a probability measure $\rho_0\in\cM_1(\R^d)$ satisfy Assumption~\aref{ass:mode} and assume that $R\geq \max\{\eps^{1/2}d^{1/4},\sqrt{2\log (1/\eps)}\}$. Consider the OU process $X$ following SDE~\eqref{eq:OU} with $\mu = 1$, started at $X_0\sim\rho_0$, and denote by $\pi_X$ its Gaussian invariant measure. Then 
\begin{align*}
\|\P_{\rho_0}(X_{T_b}\in\cdot)-\pi_X\|_{\TV}\geq (b_{\rho}-\eps)/2,&\qquad\text{where $T_b\coloneqq\log R - \log(\max\{\sqrt{2\log (1/\eps)},1\})$},\\
\|\P_{\rho_0}(X_{T_\textrm{OU}}\in\cdot)-\pi_X\|_{\TV}<\eps,&\qquad\text{where $T_\textrm{OU}\coloneqq\log R + \log (1+2\delta) + \log (1/\eps)$.}
\end{align*}
\end{prop}

\begin{rem}
\label{rem:A1-}
The mass $b_\rho$ of the furthest mode $x_0$ in~\ref{ass:mode} is typically orders of magnitude greater than the error tolerance $\eps$, making the lower bound $(b_\rho-\eps)/2$ proportional to $b_\rho$.
We note that Assumption~\ref{ass:mode} includes initial distributions $\rho_0\in\cM_{1,R,\eps}(\R^d)$ with multiple modes $x_i$ 
at distance $|x_i|\approx R$ from the origin (with the corresponding mass contained in the disc $B(x_i,\delta_i R)$) for $i\in\{1,\dots, k\}$. It follows from the proof that, in this case, the statement of Proposition~\ref{cor:cutoff} holds with a larger constant $b_\rho = \rho_0(\cup_{i=1}^kB(x_i,\delta))$.     
\end{rem}

Fixing the values of $\delta,\eps$ leads to a cut-off type phenomenon i.e., $\lim_{R\to\infty}\frac{T_\textrm{OU}}{T_b} = 1$.
The growth of the convergence time $T_\textrm{OU}$ in the distance to the furthest mode $\log R$ is  also observed in simulations.

\begin{figure}[ht]
  \centering
  \begin{subfigure}[B]{0.45\textwidth}
    \includegraphics[width=\textwidth]{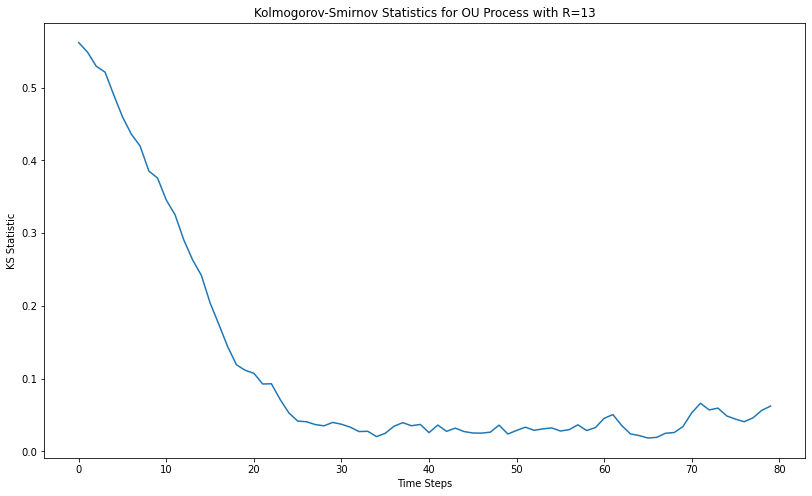}
    \caption{OU initialized at an image of dimension $d=32\times 32$ located at a distance $R=13$ from the origin.}
    \label{fig:NoCutoff}
  \end{subfigure}
  \hfill
  \begin{subfigure}[B]{0.45\textwidth}
    \includegraphics[width=\textwidth]{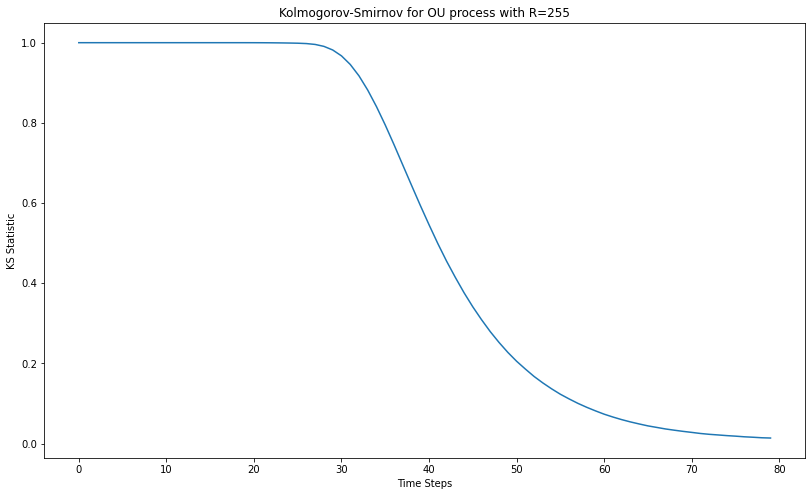}
    \caption{OU initialized at an  image  of dimension $d=1200\times1200$ located at a distance $R=255$ from the origin.}
    \label{fig:cutoff}
  \end{subfigure}
  \caption{The Kolmogorov-Smirnov (KS) test assesses whether the sample of pixels at time $t$ originates from the invariant measure $\pi_X$ of the OU process in~\eqref{eq:OU}. Since, for $\mu=1$, $\pi_X$ is a standard $d$-dimensional Gaussian, we test whether the coordinates of the marginal $X_t$ at time $t$ form a sample of $d$ independent draws from  a one dimensional Gaussian law.}
  \label{figg}
\end{figure}

The simulations in Figures~\ref{fig:NoCutoff} and~\ref{fig:cutoff} demonstrate that the convergence time for the OU process, when initialized at relevant data distributions for DDPMs, increases with the distance to the furthest mode. This observation is consistent with the results presented in Proposition~\ref{cor:cutoff}. 

The results in~\cite{MR2134115,Douc,BM2023} prove that unbounded multiplicative noise can significantly improve the convergence rate of a diffusion processes to a given stationary measure. In particular, the upper bounds in~\cite[Sec.~3.2]{MR2134115},~\cite[Sec~5.2]{Douc} and the lower bounds in~\cite[Sec.~3.1]{BM2023} on the rates of convergence imply that, for a large family of invariant measures, \textit{tempered Langevin diffusions with unbounded multiplicative noise} converge to stationarity orders of magnitude  faster than their classical Langevin counterparts (see Appendix~\ref{app:sec:A} below for more details). This fact, observed empirically and used in applications of Markov Chain Monte Carlo, naturally raises the question of whether replacing the Langevin dynamics (such as the OU process $X$) with a tempered Langevin diffusion could  shorten significantly the time horizon $T_\textrm{OU}$ in Proposition~\ref{cor:cutoff} above.
While the OU process $X$ converges to stationarity at an exponential rate, by Proposition~\ref{cor:cutoff}  the time horizon $T_\textrm{OU}$ grows with the distance $R$ of the furthest mode of the initial data distribution $\rho_0$, increasing both the discretisation and the score matching errors
(see the error bound in Equation~\eqref{eq:stable_diff_error} above).
It is hence natural to investigate whether adopting tempered Langevin dynamics with multiplicative noise could significantly decrease the time horizon when the diameter  of the initial distribution $\rho_0$ is large. 

\subsection{Tempered Langevin diffusions and non-asymptotic lower bounds} 
\label{subsec:lan}
 The role of the forward process is to transform the structured multi-modal data distribution $\rho_0$ into a ``structureless'' noise distribution $\pi\in \cM_1(\R^d)$ from which we can sample. It is thus natural to assume that the noise law  $\pi$, which is the invariant measure of the forward process, has  a twice continuously differentiable density proportional to 
$x\mapsto\exp( -h_\pi(x))$, where the function
$h_\pi:\R^d\to[0,\infty)\eqcolon\RP$. 
A natural class of forward processes for such $\pi$ are \textit{tempered Langevin diffusions}, which we now recall following~\cite{Douc}: 
for a temperature parameter $\ell\in[0,\infty)$ let
\begin{equation}
\label{eq:temperedLangevinVector}
b(x)\coloneq  -h_\pi(x)^{2\ell-1}( h_\pi(x)-2\ell) \nabla  h_\pi(x)\quad\text{\&}\quad
\sigma(x) \coloneqq \sqrt{2} h_\pi(x)^{\ell}\Id,\quad x\in\R^d,
\end{equation}
where $\Id$ is the identity matrix on $\R^d$.  Assume that the process $Y=(Y_t)_{t\in\RP}$ follows the  SDE 
\begin{equation}
\label{eq:tempLan}
    \ud Y_t = b(Y_t)\ud t + \sigma(Y_t)\ud B_t,\qquad\text{where $b$ and $\sigma$ are given in~\eqref{eq:temperedLangevinVector}}
\end{equation}
and $B$ is a standard $d$-dimensional Brownian motion. For $\ell=0$, $Y$ is a classical Langevin diffusion, with stationary measure $\pi$. For $\ell>0$, we assume that $Y$ started at an arbitrary distribution on $\R^d$  is positive recurrent and converges to its stationary measure $\pi$, see~\cite[Sec.~5.2]{Douc} and the literature therein for sufficient conditions ensuring this. 
Note that if $\pi$ is a centred Gaussian probability measure on $\R^d$ with covariance matrix $\Id$, then a tempered Langevin diffusion (with $\ell=0$) is the OU process $X$ given in~\eqref{eq:OU} with drift $\mu=1$.

In applications of DDPMs, it is key that samples from $\pi$ can be obtained efficiently. We thus assume in this section that $\pi$ is spherically symmetric,  i.e. $h_\pi(x)=H_\pi(|x|)$ for all $x\in\R^d$ and some scalar $C^2$ function $H_\pi:\RP\to[0,\infty)$. In this case, since the angular component is uniform on the unit sphere in $\R^d$, simulating samples from $\pi$ reduces to simulating its radial component on $\RP$ with density proportional to $r\mapsto \exp(-H_\pi(r))$.
If, for example, the radial density is log-concave (i.e. $H_\pi$ is convex), then exact simulation from the radial component has bounded complexity, not dependent on $H_\pi$~\cite{devroye2012note}.
This assumption simplifies the presentation but is not necessary for our general framework and results in Section~\ref{sec:main} below.

For a spherically symmetric $\pi$ and $\mu\in(0,\infty)$,
a tempered Langevin diffusion $Y$ following SDE~\eqref{eq:tempLan} is in \eqref{eq:linear_growth_spheriaclly_symmetric} if 
\begin{equation}
\label{eq:linear_growth_spheriaclly_symmetric}
\tag{$\textrm{LG}_\mu$}
H_\pi(r)^{2\ell-1}(H_\pi(r)-2\ell) H_\pi'(r)\leq \mu r\qquad\text{for all $r\in\RP$.} 
\end{equation}
Since the drift in~\eqref{eq:temperedLangevinVector} equals $b(x)= -H_\pi(|x|)^{2\ell-1}(H_\pi(|x|)-2\ell)H_\pi'(|x|)x/|x|$ for $x\in\R\setminus\{0\}$ and $b(0)=0$,  if $Y$ satisfies~\eqref{eq:linear_growth_spheriaclly_symmetric} then  the drift has  at most linear growth  $|b(x)|\leq \mu|x|$,  $x\in\R^d$.
In particular, assuming $H_\pi$ satisfies~\eqref{eq:linear_growth_spheriaclly_symmetric} for $r\in[0,1]$ and $H_\pi(r) = ar^p$ for  $r\in[1,\infty)$ with  some $p\in(0,2]$ and $a\in(0,\infty)$, condition~\eqref{eq:linear_growth_spheriaclly_symmetric} holds if
$0\leq\ell\leq 1/p-1/2$ and  $a\leq (\mu/p)^{1/(2\ell+1)}-\ell$. Hence 
the OU process $X$ in  SDE~\eqref{eq:OU}  satisfies~\eqref{eq:linear_growth_spheriaclly_symmetric}  (with $p=2$, $a=\mu/2$ and $\ell=0$).

Linking our general Assumption~\aref{ass:mode} on the initial data distribution with the forward process $Y$ satisfying~\eqref{eq:linear_growth_spheriaclly_symmetric}  and placing them in the context of DDPMs requires additional hypothesis. 
 Fix $\mu\in(0,\infty)$ and let Assumption~\aref{ass:mode} hold with $R,\eps,\delta\in(0,\infty)$ and assume there exist  $0<\beta<<1$ small, such that
\begin{equation}
    \label{eq:mode_to_dim}
    R\geq (\eps/\mu)^{1/2} d^{1/4}\qquad\text{and}\qquad R^\beta\geq 2\sqrt{\mu}(1+2\delta)/\eps,
\end{equation}
and $r_3\in(1,\infty)$ satisfying
\begin{equation}
\label{eq:pi_concentration}
\pi( B_3(r_3)\times\R^{d-3})>1-\eps/2 \qquad \text{and}\qquad 2r_3\leq R^\beta,
\end{equation}
where 
$B_3(r_3)$ is the closed ball in $\R^3$ with radius $(r_3^2-1)^{1/2}$ centered at the origin. We will discuss in Subsection~\ref{subsec:discussion} below the breadth of applicability and the role of assumptions~\eqref{eq:mode_to_dim}-\eqref{eq:pi_concentration} in Theorem~\ref{cor:error}. Our main result for tempered Langevin forward processes is as follows.

\begin{thm}
\label{cor:error} Let the data distribution $\rho_0\in\cM_{1,R,\eps}(\R^d)$ satisfy Assumption~\aref{ass:mode} with $R\in(2,\infty)$, $d\geq 3$ and $0<\eps,\delta<<1$. Assume conditions in~\eqref{eq:mode_to_dim} and~\eqref{eq:pi_concentration} hold with $\mu\in(0,\infty)$ and small $0<\beta<<1$.
For every tempered Langevin diffusion $Y$ defined via its spherically symmetric invariant measure $\pi$ through SDE~\eqref{eq:temperedLangevinVector}-\eqref{eq:tempLan} and satisfying~\eqref{eq:linear_growth_spheriaclly_symmetric}, it holds that 
$$
\|\P_{\rho_0}(Y_{T_c}\in\cdot)-\pi\|_{\TV}>(b_\rho-\eps)/2,\quad \text{where $T_c\coloneqq \frac{
1-\beta}{\mu}\log R$}.$$ 
The OU process $X$ in~\eqref{eq:OU} satisfies~\eqref{eq:linear_growth_spheriaclly_symmetric} and 
$$\|\P_{\rho_0}(X_{T_\mathrm{OU}}\in\cdot)-\pi_X\|_{\TV}<\eps, \quad \text{where $T_\mathrm{OU} \coloneqq \frac{1+\beta}{\mu}\log R$.}
$$ 
\end{thm}

\subsection{Discussion of Theorem~\ref{cor:error}, its assumptions and generalisation}
\label{subsec:discussion}
 Theorem~\ref{cor:error} shows that, for multi-modal initial distributions, replacing the OU process with a tempered Langevin diffusion with multiplicative noise does not offer a significant improvement in terms of shortening the time horizon in DDPMs. Moreover, Theorem~\ref{cor:error} rigorously establishes that, for a broad class of tempered Langevin forward processes, the time horizon necessary for DDPMs to converge (see the general error bound in~\eqref{eq:stable_diff_error} above)
 increases logarithmically in the distance $R$ to the furthest mode of the data distribution.

\paragraph{\textit{Uniform ergodicity and numerical instability.}} The main result of the paper, given in Theorem~\ref{thm:error} and Corollary~\ref{cor:main} of Section~\ref{sec:main} below,
generalises
Theorem~\ref{cor:error} by rigorously proving that the convergence horizon for a broad class of ergodic diffusions satisfying SDE~\eqref{eq:SDE_general} (without \textit{a priori} knowledge of the stationary measure $\pi$) is at least of size  $T_c$ given in Theorem~\ref{cor:error} and thus increases with the distance $R$ of the farthest mode of $\rho_0$ from the origin. Obtaining a bound on the time horizon for DDPMs that is uniform in $R$ would require 
a uniformly ergodic diffusion $Y$ with a superlinear drift  in~\eqref{eq:SDE_general}, in particular violating Assumption~\eqref{eq:linear_growth_spheriaclly_symmetric} and its general counterpart in Section~\ref{sec:main} below.  However, since  efficient sampling of uniformly ergodic diffusions with superlinear drifts  is very difficult due to numerical instabilities caused by the drift (see e.g.~\cite{Livingston19,livingstone2024skew}),  
the growth of $T_c$ as a function of  $R$, given in Theorem~\ref{cor:error},
is likely an asymptotically optimal  lower bound (in $R$) achievable via ergodic diffusions and existing sampling techniques. Thus our results  naturally motivate the exploration of stochastic interpolants~\cite{albergo2023stochastic,de2021diffusion}, potentially offering a viable alternative in addressing this challenge.

\paragraph{\textit{Forgetting the initial distribution.}}The  marginal $X_t$  of the OU process $X$ in~\eqref{eq:OU} (with $\mu \in(0,\infty)$) at time $t$  has the same law as  $ \exp(-\mu t)X_0+(1-\exp(-2t\mu))^{1/2}N/\mu^{1/2}$, where  $N$ is a standard Gaussian random vector in $\R^d$ with zero mean and covariance $\Id$, independent of $X_0\sim\rho_0$. Thus, conditional on $X_0$, $X_t-X_0\exp(-\mu t)$ quickly become a good approximation of the normal distribution with zero mean and covariance $\Id/\mu$. However, if the support of the initial data distribution $\rho_0$ is large, i.e., $\rho_0\in\cM_{1,R,\eps}(\R^d)$ with $R>>1$, it takes  the marginal $X_t$, averaged over $\rho_0$,  at least $t=(\log R)/\mu$ to forget the initial distribution and resemble a standard normal globally. To leverage this fact in applications, practitioners design algorithms to take tiny time steps initially and larger time steps subsequently, see~\cite[Thm~2]{chen2023improved} and~\cite[Sec.~2.4.1]{conforti2023score} and the references therein. Theorem~\ref{cor:error} and its generalisation in Theorem~\ref{thm:error} of Section~\ref{sec:main} demonstrate that an analogous phenomenon persists for all diffusions with drift that is not superlinear, as the time required for forgetting the initial distribution cannot be significantly reduced.

Finally, we note that increasing the constant $\mu$ in Assumption~\eqref{eq:linear_growth_spheriaclly_symmetric} would decrease the time necessary to forget the initial data distribution. However, as is  well know among practitioners, this would not improve the performance of the corresponding DDPM. This is due to fact that the increase in $\mu$ amounts to a deterministic time change of the forward process, which in turn requires smaller time steps in the simulation stage of the DDPM. It is thus natural to restrict the class of forward diffusions in Theorem~\ref{cor:error} to~\eqref{eq:linear_growth_spheriaclly_symmetric} for a fixed value of $\mu$.

\paragraph{\textit{Assumptions~\eqref{eq:mode_to_dim} and~\eqref{eq:pi_concentration} in Theorem~\ref{cor:error}}} The first inequality in~\eqref{eq:mode_to_dim}, $R\geq (\eps/\mu)^{1/2} d^{1/4}$, stipulates that the distance $R$ to the furthest mode  of the initial data distribution $\rho_0$ grows with dimension $d$ and that the growth rate of the drift $\mu$ is not too small  (for image data sets encountered in applications,
$R$ is typically proportional to $d^{1/2}$, see Section~\ref{subsec:Applications} below for more details).
The second inequality in~\eqref{eq:mode_to_dim}, $R^\beta\geq 2\sqrt{\mu}(1+2\delta)/\eps$, essentially requires that, for some small $\beta$, the $\beta$-power of the distance to the furthest mode exceeds $1/\eps$, where $\eps$ is the error tolerance. The assumption is natural, since $\eps$ is typically fixed and $R$ grows polynomially in dimension.  


The first inequality $\pi( B_3(r_3)\times\R^{d-3})>1-\eps/2$ in~\eqref{eq:pi_concentration} holds for any probability law $\pi$ if $r_3$ is sufficiently large. Thus the content of Assumption~\eqref{eq:pi_concentration} lies in the restriction imposed by the second inequality $2r_3\leq R^\beta$. A natural question here is how $r_3$ grows with dimension $d$ for relevant noise distributions $\pi$.
Clearly, if $\pi$ is a centered Gaussian probability measure on $\R^d$, the constant $r_3$ does not depend on $d$. Another natural choice for $\pi$ is given by the generalized Laplace distributions on $\R^d$, which can be represented as  a fixed time marginal of a standard $d$-dimensional Brownian motion subordinated by a Gamma subordinator~\cite{MR2984356}. In this case, the projection  of $\pi$ onto the first three coordinates also does not depend on $d$, again making the constant $r_3$ depend only on the error tolerance $\eps>0$, uniformly across all dimensions. For distributions $\pi$ with tails asymptotic to  $x\mapsto  \exp(-a|x|^p+b\log |x|)$, for $p\in[1,2)$, $a>0$ and $b\in[0,\infty)$, as $|x|\to\infty$, the growth of the quantile $r_3$ depends on the values of the parameters. The representation of marginal densities for spherically symmetric distributions in $\R^d$ in~\cite[Eq.~(1.4)]{Kotz02} suggests that, for $b\in((d-3)/2,\infty)$, $r_3$ depends only on the error tolerance $\eps>0$ and not on dimension. Moreover, when $b=0$ and $p\in(1,2)$ the growth of $r_3$ suggested by the simulation in Appendix~\ref{app:dependence_of_r_3} appears to be logarithmic in $d$. In contrast, the distance to the furthest mode $R$ is typically proportional to $d^{1/2}$, making condition~\eqref{eq:pi_concentration} valid for a large classes of invariant measures $\pi$ discussed above. 

\paragraph{\textit{How is Theorem~\ref{cor:error} proved?}} 
The main result of the paper, Theorem~\ref{thm:error} below, 
states lower bounds for a general class of ergodic diffusions, which includes  tempered Langevin processes, and essentially implies Theorem~\ref{cor:error}. The main idea of the proof of Theorem~\ref{thm:error} is inspired by~\cite{MR2540073} and outlined in Section~\ref{subsubSec:main_proof_idea} below.
The generalisation of Assumption~\eqref{eq:pi_concentration} to ergodic diffusions  requires projections onto $k$-dimensional subspaces, where $k$ is  typically much smaller than $d$ but at least $3$. It will become clear from the proof of Theorem~\ref{thm:error} in Section~\ref{sec:proofs} below, that the reason why $3$ dimensions suffice for tempered Langevin diffusions is due to the fact that $\sigma$ in~\eqref{eq:temperedLangevinVector} is a scalar function. 

We note that the proof of Theorem~\ref{thm:error} allows for the mass of multiple modes of $\rho_0$ at the maximal distance from the origin to be included in  $b_\rho$, thus improving 
the lower bound of Theorem~\ref{cor:error} (see 
Remark~\ref{rem:A1-} above for a formal description of this phenomenon).  

\section{A general framework for forward processes}
\label{sec:main}
In this section we consider solutions to a general elliptic SDE and state a generalisation to Theorem~\ref{cor:error} in this broader context. Let $b:\R^d\to\R^d$ and $\sigma:\R^d\to\R^{d\times d}$. Consider a unique solution $Y$ of the SDE
\begin{equation}
    \label{eq:SDE_general}
    \ud Y_t = b(Y_t)\ud t + \sigma(Y_t)\ud B_t,
\end{equation}
where $B$ is a $d$-dimensional Brownian motion and let $Y$ admit an invariant measure $\pi$. The following assumption on the drift and dispersion coefficients of $Y$ plays a key role in our main result (Theorem~\ref{thm:error} below).  

\subsection{Main result} In this section we describe the class of diffusions that could be used as forward processes in DDPMs, give our main theorem and discuss key ideas behind its proof. 

 \noindent \textbf{Assumption} \namedlabel{ass:Lip}{{\color{purple}\texttt{(ForProc)}}}
 \textit{Let $\mu\in(0,\infty)$, $d\geq 3$ and $\eps>0$. Consider
 the diffusion $Y$ satisfying SDE~\eqref{eq:SDE_general}. 
 Assume drift $b$ exhibits at most linear growth in each direction:}
\begin{equation}
\label{eq:linear_growth}
 |\langle b(x),u/|u|\rangle|\leq \mu|\langle x,u/|u|\rangle| \quad \text{for all $x\in\R^d$ and $u\in\R^d\setminus\{0\}$.}
\end{equation}
 \textit{For some $k\in\{3,\dots,d\}$ and any $y_1\in\R^d$ with $|y_1|=1$, there exist orthonormal vectors $y_1,\dots,y_k\in\R^d$ 
 such that the orthogonal projection 
 $G_{\mathcal{Y}}:\R^d\to \mathcal{Y}$, $G_\mathcal{Y}(x)\coloneqq \sum_{j=1}^k y_k\langle y_k,x\rangle$,
$x\in\R^d$,  onto  the 
 vector subspace $\mathcal{Y}$ spanned by $\{y_1,\ldots,y_k\}$
 and the dispersion $a = \sigma\sigma^\intercal$
satisfy}
\begin{equation}
\label{eq:sigma_bound}
\sum_{j=1}^k\langle  ay_j,y_j\rangle\geq 3\langle a\hat G_\mathcal{Y},\hat G_\mathcal{Y}\rangle\quad \text{on $\R^d$, where $\hat G_\mathcal{Y}\coloneqq  H_\mathcal{Y} G_\mathcal{Y}$ and $H_\mathcal{Y}:=(1+|G_\mathcal{Y}|^2)^{-1/2}$.}
\end{equation}
\textit{
Let $\pi$ be the invariant measure of $Y$ and pick  $r_k\in(1,\infty)$ such that}
\begin{equation}
\label{eq:f_moment}
\pi(\{x\in\R^d:|G_\mathcal{Y}(x)|^2+1\leq r_k^2\})>1-\eps/2.
\end{equation}

\begin{rem}
\label{rem:Lip} 
Assumption~\ref{ass:Lip} is satisfied for a wide range of diffusion processes, including the tempered Langevin diffusions from Section~\ref{subsec:lan} above. 
    The linear growth condition in~\eqref{eq:linear_growth} clearly holds under~\eqref{eq:linear_growth_spheriaclly_symmetric} for tempered Langevin diffusions with spherically symmetric invariant measures. Tempered Langevin diffusions in SDE~\eqref{eq:tempLan} also satisfy condition~\eqref{eq:sigma_bound} 
    with $k=3$. Indeed, 
    for any diagonal dispersion matrix $a= \sigma\sigma^\intercal$, 
    \eqref{eq:sigma_bound} holds if for some $k\in\{3,\dots,d\}$,
    \begin{equation*}
    k\min_{j\in\{1,\dots,d\}}  \langle a(x)e_j,e_j\rangle\geq 3\max_{j\in\{1,\dots,d\}}  \langle a(x)e_j,e_j\rangle \quad\text{for all $x\in\R^d$}.
    \end{equation*}
    Since all the diagonal elements of $a= \sigma\sigma^\intercal$ in SDE~\eqref{eq:tempLan} are equal,  this inequality holds with $k=3$. 
    The factor $3$ in~\eqref{eq:sigma_bound} allows us to obtain a bound on the Laplacian of a Lyapunov function used in the proof of Theorem~\ref{thm:error}. In the setting of general diffusions following SDE~\eqref{eq:SDE_general},  concentration of the projection of the stationary measure $\pi$ onto $k$-dimensional subspaces is required (in the tempered Langevin case, $3$-dimensional subspace sufficed), necessitating the introduction of the corresponding quantile in~\eqref{eq:f_moment}.
    In particular, all tempered Langevin diffusions studied in Section~\ref{subsec:lan}, including the OU process following SDE~\eqref{eq:OU}, satisfy Assumption~\aref{ass:Lip}.

    We note that a  diagonal dispersion matrix $a(x)$ 
    with all but one diagonal elements bounded in $x\in\R^d$,  violates~\eqref{eq:sigma_bound} (choose $y_1$ to be the eigenvector corresponding to the unbounded diagonal element).
Thus condition~\eqref{eq:sigma_bound} can be viewed as a requirement on the dispersion coefficient $a$ to be balanced across various directions,  generalising the class of tempered Langevin diffusions of Section~\ref{subsec:lan} above (where the dispersion coefficients scales all directions equally at every point of the state space $\R^d$). 
\end{rem}

Assumption~\aref{ass:Lip} allows us to consider a wide range of ergodic elliptic diffusions  as potential forward processes in DDPMs (see e.g.~\cite{BM2023,Douc,khasminski} for numerous models satisfying~\aref{ass:Lip}). 

\begin{thm}
\label{thm:error} Let the data distribution $\rho_0\in\cM_{1,R,\eps}(\R^d)$ and the diffusion $Y$ following SDE~\eqref{eq:SDE_general} satisfy~\ref{ass:mode} and~\ref{ass:Lip}, respectively, for some $0<\eps,\delta,\beta<<1$ and $R,r_k,\mu\in(0,\infty)$ with $2r_k\leq R$. Then we have
$$
\|\P_{\rho_0}(Y_{T_c}\in\cdot)-\pi\|_{\TV}>(b_\rho-\eps)/2>>\eps,\quad\text{where $T_c\coloneqq  \frac{1}{\mu}\log \left(\frac{R}{2r_k}\right)$}.
$$ 
The OU process $X$ in SDE~\eqref{eq:OU} satisfies~\aref{ass:Lip} (with Gaussian stationary law $\pi_X$ with zero mean and covariance $\Id/\mu$) and the following inequality holds:
$$\|\P_{\rho_0}(X_{T_\mathrm{OU}}\in\cdot)-\pi_X\|_{\TV}<\eps,\quad \text{where $T_{\mathrm{OU}} \coloneqq \frac{1}{\mu}\max\{\log (2d^{1/4}/\eps^{1/2}),\log (2R(1+2\delta)\sqrt{\mu}/\eps)\}$.}$$ 
\end{thm}

Comparison between forward processes in the context of DDPMs requires non-asymptotic bounds. Such non-asymptotic bounds on the convergence of Markov processes have been extensively studied~\cite{Durmus17,Andrieu22,bakry08}, particularly motivated by applications in MCMC. A common approach for establishing such bounds is based on Poincar\'e inequalities~\cite{Durmus17,Andrieu22,bakry08}. However, this powerful method often necessitates strong assumptions on both the initial condition $\rho_0$ and the transition kernel of the process. Such assumptions are typically not satisfied by the general diffusion process satisfying~\ref{ass:Lip} and initial conditions distributions under the manifold hypothesis. By focusing on \textit{lower} rather than \textit{upper} bounds, we obtain non-asymptotic bounds in Theorem~\ref{thm:error} under mild assumptions using ideas from~\cite{MR2540073,BM2023}. Our results (see Lemma~\ref{lem:markov_bound} below for more details) on the convergence of Markov processes are independent of Assumptions~\ref{ass:mode} and \ref{ass:Lip} and can be applied in other settings. Given that upper bounds on the OU process can be explicitly computed, this approach enables us to make a direct non-asymptotic comparison, showing that the OU process is hard to beat in context of DDPMs.

\begin{rem}
\label{rem:thm_main_comment}
Theorem~\ref{thm:error} provides a lower bound on the convergence to stationarity for a large family of ergodic diffusions using a novel approach. In Proposition~\ref{cor:cutoff}, this lower bound for the OU process is calculated using the explicit form of marginal densities. It is natural to compare  the lower bound in Theorem~\ref{thm:error}, applied directly to the OU process, with the one from Proposition~\ref{cor:cutoff}. According to Proposition~\ref{cor:cutoff}, the OU process $X$ with the parameter $\mu=1$  does not converge before time $T_b= \log R-\log(\max\{\sqrt{2\log (1/\eps)},1\})$. Since the OU process belongs to the class of tempered Langevin diffusions, we can work with three-dimensional projections.  With direct calculation, we can bound the $\eps$-quantile of the three-dimensional Gaussian by $r_3=\sqrt{4\log (1/\eps)}$. Thus, Theorem~\ref{thm:error} shows that the OU does not converge before time $T_c = \log R -\log (2\sqrt{6\log (1/\eps)})$. This demonstrates that the lower bound in Theorem~\ref{thm:error} provides a good approximation. In particular, the bound is sharp in cases where $R$ is large, which is relevant for high dimensional initial distributions.

Note that the time $T_{\mathrm{OU}}$ in Theorem~\ref{thm:error} differs from time $T_{\mathrm{OU}}$ in Proposition~\ref{cor:cutoff}. This difference is due to the fact that in Proposition~\ref{cor:cutoff} we have assumed $R \geq \eps^{1/2}d^{1/4}$, and in this case, $T_{\mathrm{OU}}$ from Theorem~\ref{thm:error} reduces to the form presented in Proposition~\ref{cor:cutoff}.
\end{rem}

\subsubsection{Discussion of the proof of Theorem~\ref{thm:error}} \label{subsubSec:main_proof_idea}
The proof of Theorem~\ref{thm:error} rests on an idea inspired by~\cite{MR2540073}. A necessary condition for the convergence of the process $Y$, initialized with a multi-modal distribution $\rho_0$, to its unimodal invariant measure $\pi$ involves the transport of mass from all modes of $\rho_0$ towards the origin. Thus, comparing the mass that the invariant measure $\pi$ and the marginal distribution $Y_T$ place around the origin should provides a good lower bound on the total variation distance between them. The inspiration for this approach comes from~\cite{MR2540073}, where lower bounds on the convergence of certain
hypoelliptic diffusions to their heavy-tailed invariant measure was studied. In contrast to our situation, in~\cite{MR2540073} the  transport of mass from the ``centre'' of the space to the tails  played a key role, with critical sets (yielding sharp lower bounds) 
being  complements of large compacts. In our setting an ``inverse'' of this idea is used with critical sets being compact and centred around the origin 
as depicted in Figure~\ref{fig:Evolution}.

\begin{figure}[hbt]
\centering
    \includegraphics[width=120mm]{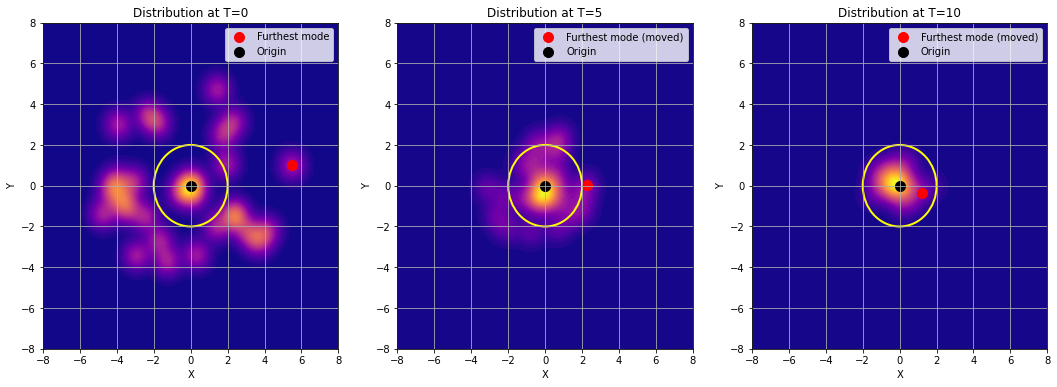}
    \caption{Marginals of an OU process $X$ initialised at a distribution $\rho_0$ depicted by the heat map in the first panel. The heat maps in the second and third panel depict the marginals $X_5$ and $X_{10}$. The yellow circle represents the boundary of the disc $B(0,2)$ of radius $2$, centred at the origin.}
    \label{fig:Evolution}
\end{figure}

An important aspect of the proof of Theorem~\ref{thm:error} is that, due to the high dimensionality of the problem, a direct comparison of the mass near the origin of the invariant measure $\pi$ and the marginal law of $Y_T$ is no longer effective. This issue arises
even if $\pi$ is a standard Gaussian distribution
because the radius 
 $r_d$, satisfying $\pi(B(0,r_d))>1-\eps$, grows as $r_d\approx \sqrt{d}$ yielding vanishingly small lower bounds on the total variation distance between $\pi$ and the law of $Y_T$.\footnote{Note that $\pi(\R^d\setminus B(0,r_d))$ equals 
 the square root of the quantile of the  $\chi^2(d)$-distribution with $d$ degrees of freedom. Quantiles of the $\chi^2(d)$-distribution are known to be proportional to $d$ as $d\to\infty$, see e.g.~\cite[p.~426]{Johnson94}.}  This is in contrast to the problem in~\cite{MR2540073}, where the dimension of the physical model is fixed and the focus is on the tails. In order to apply the ``inverse'' of the idea in~\cite{MR2540073}, we work with projections onto lower-dimensional linear subspaces, containing the furthest modes of $\rho_0$ at distance $R$ from the origin (under Assumption~\ref{ass:Lip}, the vector $y_1$ is chosen to point towards the furthest mode of $\rho_0$).  If the initial data distribution $\rho_0$ has several modes approximately at distance $R$ away from the origin (see Remark~\ref{rem:A1-} above for a formal description), the proof of Theorem~\ref{thm:error} yields a greater lower bound with $b_\rho$ equaling the total mass of appropriate discs centered at all these modes (instead of taking $y_1$, pointing to the furthest mode, consider an orthonormal basis of the vector subspace generated by the modes at distance approximately $R$). 
 
 Finally we note that the projections on lower-dimensional subspaces are consistent with the manifold hypothesis, which states that complex high-dimensional data is supported on lower-dimensional submanifolds of the Euclidean space~\cite{de2022convergence, Fefferman16}. 

\subsection{Application of Theorem~\ref{thm:error}}
\label{subsec:Applications}
Linking  the general assumptions~\aref{ass:mode} and~\aref{ass:Lip} on the data distribution and the forward process, respectively, and placing them  in the context of stable diffusion algorithms, requires an additional assumption~\aref{ass:Add} below. Its two main aims are as follows: first, it restricts the subset of initial data distributions in~\aref{ass:mode}, to the ones more accurately reflecting the distributions used in practical applications; second, it reduces the class of ergodic diffusions in~\aref{ass:Lip} to the family of exponentially ergodic processes with light-tailed invariant measures. As we will see, both of these features of~\aref{ass:Add} are natural from the point of view of applications of DDPMs. 

\smallskip

\noindent \textbf{Assumption} \namedlabel{ass:Add}
{{\color{purple}\texttt{(DATA)}$\leftrightarrow$\texttt{(ForProc)}}}
\textit{Let~\aref{ass:mode} hold with some $R,\eps,\delta\in(0,\infty)$.
Let~\aref{ass:Lip} hold with the same $\eps$ and some $\mu\in(0,\infty)$. Assume the following holds for $0<\beta<<1$ small.}
\begin{enumerate}[label=(\alph*)]
    \item \label{ass:add:log_dimension} \textit{Distance $R$ of the furthers mode of $\rho_0$ grows with dimension as  $R\geq (\eps/\mu)^{1/2}d^{1/4}$. Moreover, the radius of the mode and the error tolerance are small in comparison to the distance to the mode $R^\beta\geq 2\sqrt{\mu}(1+2\delta)/\eps$.}
    \item \label{ass:add_diff} \textit{For $\eps$ in Assumption~\aref{ass:mode}, condition~\eqref{eq:f_moment} holds with $r_k\in(1,\infty)$ satisfying $2r_k\leq R^\beta$.}
\end{enumerate}

In application such as  image processing, each color channel is encoded in a range, typically
$[0, 255]$ for each channel. The largest distance between modes, which is proportional to $R$,  is given by the contrast between bright and dark images is typically proportional to $d^{1/2}$, where $d$ represents the dimension. This fact has been used in~\cite{de2021diffusion} to validate the assumption that initial data distribution admits a compact support. In this work we go further by \textit{quantifying} the diameter of the aforementioned support. Since the canonical choice for the drift parameter in applications is $\mu = 1$~\cite{conforti2023score,benton2024nearly,chen2022sampling}, Assumption~\aref{ass:Add}\ref{ass:add:log_dimension} holds in this context.

Assumption~\aref{ass:Add}\ref{ass:add_diff} is concerned with the lightness of tails of the invariant measure $\pi$, of the forward process $Y$. In particular, it stipulates that the quantiles of the $k$-dimensional projection of $\pi$ increase more  slowly than the polynomial $R^\beta$ of the distance to the furthest mode. As noted in the previous paragraph,  this quantity is often  proportional to $d^{\beta/2}$. As discussed in Section~\ref{subsec:discussion}, condition~\aref{ass:Add}\ref{ass:add_diff} is satisfied for a large family of spherically symmetric distributions on $\R^d$.

\begin{cor}
\label{cor:main} Let the data distribution $\rho_0\in\cM_{1,R,\eps}(\R^d)$ and the diffusion $Y$ following SDE~\eqref{eq:SDE_general} satisfy~\aref{ass:Add} with parameters $R,\mu\in(0,\infty)$ and $0<\eps,\delta,\beta<<1$. Then we have
$$
\|\P_{\rho_0}(Y_{T_c}\in\cdot)-\pi\|_{\TV}>(b_\rho-\eps)/2>\eps,\quad \text{where $T_c\coloneqq \frac{
1-\beta}{\mu}\log R$}.$$ The OU process $X$ in~\eqref{eq:OU} satisfies~Assumption~\aref{ass:Add}  and 
$$\|\P_{\rho_0}(X_{T_\mathrm{OU}}\in\cdot)-\pi_X\|_{\TV}<\eps, \quad \text{where $T_\mathrm{OU} \coloneqq \frac{1+\beta}{\mu}\log R$}.
$$ 
\end{cor}
 
\section{Proofs}
\label{sec:proofs}

Consider a Markov process $\kappa=(\kappa_{t})_{t\in\RP}$ taking values on $\R^d$. Following the monograph~\cite[Ch~1, Def~(14.15)]{Davis}, let $\cD(\cA)$ denote the set of measurable functions $g:\R^n \to \R$ with the following property: there exists a measurable $h:\R^d\to\R$, such that, for each $x\in\R^d$, $t\rightarrow h(\kappa_t)$ is integrable $\P_x$-a.s. 
 and the process $$
g(\kappa)-g(x)-\int_0^\cdot h(\kappa_s)\ud s\quad\text{is a $\P_x$-local martingale.}
$$
 Then we write $h = \cA g$ and call $(\cA,\cD(\cA))$ the \textit{extended generator} of the process $\kappa$. The first step in the proof of Theorem~\ref{thm:error} is the following lemma, which provides a lower bound on the total variation distance between a marginal and an invariant measure of the general Markov process $\kappa$.

\begin{lem}
\label{lem:markov_bound}
Let $\kappa=(\kappa_{t})_{t\in\RP}$ be a Markov process with an extended generator $\cA$ and an invariant measure $\pi$. Assume that for $H:\R^d\to(0,1]$ in $\cD(\cA)$ and a concave, increasing, differentiable function $\xi:(0,\infty)\to(0,\infty)$ we have  $\cA H\leq \xi \circ H$ on $\R^d$. 
Define $$\Xi:\{(u,v)\in[0,\infty)\times(0,\infty]:u\leq v\}\to[0,\infty]\quad\text{ by }\quad
\Xi(u, v)\coloneqq \int_u^v\ud s/\xi(s).$$ 
For every $u_0\in(0,1]$ there exists a unique function  $\gamma(u_0,\cdot):[0,\Xi(u_0,\infty))\to [u_0,\infty)$ satisfying $\Xi(u_0,\gamma(u_0,y))=y$ for all $y\in[0,\Xi(u_0,\infty))$. For any fixed $y_0\in(0,\Xi(1,\infty))$, the function 
$\gamma(\cdot,y_0):(0,1]\to[0,\infty)$ is increasing.  By defining $\gamma(0,y_0)\coloneqq \lim_{u\downarrow0}\gamma(u,y_0)\in\RP$, 
we extend the function $\gamma(\cdot,y_0)$ to $[0,1]$.
There exists a unique increasing function $\eta_{y_0}:[\gamma(0,y_0),\gamma(1,y_0))\to(0,1]$ satisfying $\gamma(\eta_{y_0}(s),y_0)=s$ for $s\in[\gamma(0,y_0),\gamma(1,y_0))$.

Fix  $\rho_0\in\cM_1(\R^n)$. For every $T\in(0,\Xi(1,\infty))$ and $r\in[1,\infty)$, satisfying $\gamma(0,T)\leq 1/r<\gamma(1,T)$, define $C_{r,T}\coloneqq \eta_T(1/r)$.  Then, the following inequality holds
\begin{equation}
\label{eq:main_bound}
\|\P_{\rho_0}(\kappa_T\in \cdot)-\pi\|_{\TV}\geq \pi(\{H\geq 1/r\})-\rho_0(H\geq C_{r,T})-\int_{\{H< C_{r,T}\}} r\gamma(H(x),T)\rho_0(\ud x).
\end{equation}
\end{lem}

\begin{proof} 
The existence and uniqueness of the function 
$\gamma(u_0,\cdot)$, satisfying $\Xi(u_0,\gamma(u_0,y))=y$ for all $y\in[0,\Xi(u_0,\infty))$, follows from the monotonicity of the function $v\mapsto\Xi(u_0,v)$. Since  $\Xi(u,\infty)>\Xi(1,\infty)$ for $u\in(0,1)$, for any $y_0\in(0,\Xi(1,\infty))$ the function  $u\mapsto \gamma(u,y_0)$ is well defined on the interval $(0,1]$. Moreover, 
$\gamma(\cdot,y_0)$ is increasing (since $u\mapsto \Xi(u,v)$ is decreasing on $(0,v]$ for every $v\in(0,\infty)$), making its inverse $\eta_{y_0}$ also increasing.

We now state the following fact, which is established below, and use it to prove inequality~\eqref{eq:main_bound}.

\noindent\textbf{Claim.} For each $x\in\R^d$, the inequality $\E_x[H(\kappa_t)]\leq \gamma(H(x),t)$ holds for all $t\in(0,\Xi(H(x),\infty))$.

In order to prove the lemma using the claim, fix $T\in (0,\Xi(1,\infty))$ and $r\geq1$ as in the statement of the lemma. 
Since $T\leq \Xi(1,\infty)\leq \Xi(H(x),\infty)$, the claim and Markov's inequality imply that for every $x\in\R^d$  we have
\begin{align*}
\P_x(H(\kappa_T)\geq 1/r)\leq r\E_x[H(\kappa_T)]\leq r\gamma(H(x),T).
\end{align*}
Recall that $\eta_T$ denotes the inverse of the increasing function $u\mapsto \gamma(u,T)$ and $C_{r,T} = \eta_T(1/r)$. Thus for $x\in\{H< C_{r,T}\}$ we have $r\gamma(H(x),T) \leq r\gamma(\eta_T(1/r),T)\leq 1$, yielding
\begin{align*} 
\|\P_{\rho_0}(\kappa_T\in \cdot)-\pi\|_{\TV} & \geq \pi(\{H\geq  1/r\})- \P_{\rho_0}(H(X_T)\geq 1/r)\\
& \geq \pi(\{H\geq 1/r\})- \E_{\rho_0}[\1{H(\kappa_T)\geq 1/r}(\1{H(\kappa_0)\geq C_{r,T}})\\
& \quad +\1{H(\kappa_0)<C_{r,T}})]\\
& \geq \pi(\{H\geq 1/r\}-\rho_0(H\geq C_{r,T})-r\E_{\rho_0}[H(\kappa_T)\1{H(\kappa_0)< C_{r,T}})]\\
&\geq \pi(\{H\geq 1/r\})-\rho_0(H\geq C_{r,T})-\int_{\{H< C_{r,T}\}} r\gamma(H(x),T)\rho_0(\ud x),
\end{align*}
where the last inequality follows from the claim. It remains to establish the claim.

\noindent \underline{Proof of Claim.} Pick  $x\in\R^d$ and recall  $H\in\cD (\cA)$.
Thus, by~\cite[Ch~1, Def~(14.15)]{Davis}, there exists an increasing sequence $\{T_n:n\in\N\}$ of stopping times, such that
$T_n\uparrow \infty$ as $n\to\infty$ $\P_x$-a.s., 
and  $M^{(n)}\coloneqq H(\kappa_{\cdot \wedge T_n})-H(x) - \int_0^{\cdot\wedge T_n} \cA H(\kappa_s)\ud s$ is a martingale under $\P_x$ for all $n\in\N$. Since, for every $m\in\N$, we have $\E|M^{(m)}_t|<\infty$ and $H(\kappa_{t\wedge T_m})$ is bounded, we have $\E_x[|\int_0^{t\wedge T_m} \cA H(\kappa_s)\ud s|]<\infty$. Moreover, since $\cA H\leq \xi \circ H$ on $\R^d$, for any $t\in\RP$ and $m\in\N$ we obtain
\begin{align*}
\E_x[H(\kappa_{t\wedge T_m})] = H(x) + \E_x\left[\int_0^{t\wedge T_m}\cA H(\kappa_s)\ud s\right] \leq H(x) + \E_x\left[\int_0^{t\wedge T_m} \xi\circ H(\kappa_s)\ud s\right].
\end{align*}
This inequality, Fatou's lemma (applicable since $H:\R^d\to(0,1]$ is bounded from below) and the monotone convergence theorem (applicable since $\xi>0$) yield
\begin{align*}
\E_x[H(\kappa_t)] &=\E_x[\liminf_{m\to\infty} H(\kappa_{t\wedge T_m})]\leq \liminf_{m\to\infty}\E_x[H(\kappa_{t\wedge T_m})]\\
&\leq H(x) + \liminf_{m\to\infty}\E_x\left[\int_0^{t\wedge T_m} \xi \circ H(\kappa_s)\ud s\right] = H(x) + \E_x\left[\int_0^{t} \xi \circ H(\kappa_s)\ud s\right]. 
\end{align*}
Since $\xi:(0,\infty)\to(0,\infty)$ is concave, Tonelli's theorem and Jensen's inequality imply $$
\E_x\left[\int_0^{t} \xi \circ H(\kappa_s)\ud s\right]=\int_0^t \E_x\left[\xi\circ H(\kappa_s)\right]\ud s\leq \int_0^t \xi\left(\E_x[ H(\kappa_s)]\right)\ud s\quad \text{for all $t\in\RP$.}
$$
For $t\in\RP$, let $g(t)\coloneqq\E_x[H(\kappa_t)]$. Then, we have $g(0)=H(x)$ and 
\begin{equation}
\label{eq:generator_bound_growth}
g(t)\leq g(0)+G(t)\quad \text{for all $t\in\RP$, where $G(t)\coloneqq \int_0^t\xi(g(s))\ud s$.}
\end{equation}
Since  $\xi$ is increasing, by~\eqref{eq:generator_bound_growth} we obtain
\begin{equation}
\label{eq:derrivtive_Version}
G'(v)=\xi(g(v))\leq \xi(g(0)+G(v))\qquad\text{for all $v\in\RP$.}
\end{equation}
Recall that the increasing function $\Xi(H(x),\cdot):[g(0),\infty)\to[0,\Xi(g(0),\infty))$, given by the integral $\Xi(H(x),t) =\int_{g(0)}^t \ud s/\xi(s)$, has a differentiable inverse $\gamma(H(x),\cdot)$.    Since $G$ is increasing,  the change of variables $z=G(v)$ in the following integral and the bound in~\eqref{eq:derrivtive_Version} yield
\begin{align*} 
\Xi(H(x),g(0)+G(t)) &=
\int_{g(0)}^{g(0)+G(t)} \frac{\ud z}{\xi(z)}=
\int_0^{G(t)} \frac{\ud z}{\xi(g(0)+z)}\\
&=\int_0^t \frac{G'(v)}{\xi(g(0)+G(v))}\ud v
 \leq t=\Xi(H(x),\gamma(H(x),t)).
\end{align*}
Hence,
by~\eqref{eq:generator_bound_growth}, we get
$g(t)\leq g(0)+G(t)\leq \gamma(H(x),t)$ for all $t\in[0,\Xi(H(x),\infty))$, proving the claim.
\end{proof}

Let $Y$ be the solution of the SDE in~\eqref{eq:SDE_general}. By It\^o's formula applied to the process $g(Y)$, the extended generator of the diffusion $Y$ takes the following form for any twice continuously differentiable $g:\R^d\to\R$:
\begin{equation}
\label{eq:generator}
\cA g(x) = \langle b(x),\nabla g(x)\rangle + \frac{1}{2}\Tr(a(x)\Hessian(g)(x))\quad\text{for all $x\in\R^d$, where  $a=\sigma\sigma^\intercal$,}
\end{equation}
$\Hessian(g)$ is a symmetric matrix in $\R^{d\times d}$ consisting of the second order partial derivatives of $g$ and the trace operator $\Tr(\cdot)$ returns the sum of the diagonal elements of a square matrix. 

The next proposition controls the infinitesimal expected growth rate of the It\^o process $H_\mathcal{Y}(Y)$ for a forward diffusion $Y$ satisfying Assumption~\aref{ass:Lip}. This result concerns the infinitesimal behaviour of $Y$ only (i.e. does not depend on the initial law $\rho_0$) and will play a key role in the proof of Theorem~\ref{thm:error} given below.   
Recall, by~\aref{ass:Lip}, that for any $y_1\in\R^d$ with $|y_1|=1$ there exists $k\in\N$ and orthonormal vectors $y_1,\dots,y_k\in\R^d$, spanning the vector space $\mathcal{Y}$, such that the dispersion matrix $a=\sigma\sigma^\intercal$ satisfies the inequality in~\eqref{eq:sigma_bound}. Recall also that 
 $G_\mathcal{Y}(x)=\sum_{j=1}^k y_j \langle x,y_j\rangle$  defined in~\ref{ass:Lip} denotes the orthogonal projection of $\R^d$ onto $\mathcal{Y}$
 and that the twice continuously differentiable function $H_\mathcal{Y}:\R^d \to(0,1]$ in~\eqref{eq:sigma_bound} is given by the formula 
 $H_\mathcal{Y}(x) \coloneqq \left(1+|G_\mathcal{Y}(x)|^2\right)^{-{1/2}}$. 
 In particular, we have $|G_\mathcal{Y}(x)|^2=\sum_{j=1}^k \langle x,y_j\rangle^2$ and $|G_\mathcal{Y}(x)| H_\mathcal{Y}(x)\leq 1$ for all  $x\in\R^d$.

\begin{prop}
\label{prop:generator_bound}
Let a diffusion $Y$ satisfy Assumption~\aref{ass:Lip} with some $\mu\in(0,\infty)$. Then the extended generator $\cA$ of $Y$, given in~\eqref{eq:generator} above, applied to the function $H_\mathcal{Y}$ satisfies  $\cA H_\mathcal{Y}\leq \mu H_\mathcal{Y}$ on $\R^d$.
\end{prop}

\begin{proof}
For all $x\in\R^d$ it holds that $\nabla H_\mathcal{Y}(x)= -H_\mathcal{Y}(x)^3G_\mathcal{Y}(x)$. Moreover the Hessian of $H_\mathcal{Y}$ takes the form  $\Hessian(H_\mathcal{Y})(x)=H_\mathcal{Y}(x)^3(3H_\mathcal{Y}(x)^2 G_\mathcal{Y}(x) G_\mathcal{Y}(x)^\intercal-\sum_{j=1}^k y_jy_j^\intercal)$, $x\in\R^d$, where, for any $y\in\R^d$, the transposition $y^\intercal$ turns $y$ into a row vector with the same entries. It follows that
\begin{align*}
\Tr(a(x)\Hessian(H_\mathcal{Y})(x)) &= H_\mathcal{Y}(x)^3\left(3H_\mathcal{Y}(x)^2\langle a(x) G_\mathcal{Y}(x),G_\mathcal{Y}(x)\rangle-\sum_{j=1}^k\langle a(x) y_j,y_j\rangle\right) \leq 0
\end{align*}
for all $x\in\R^d$, since the dispersion matrix $a=\sigma\sigma^\intercal$ satisfies~\eqref{eq:sigma_bound} 
in Assumption~\aref{ass:Lip} above. We conclude that $\Tr(a\Hessian(H_\mathcal{Y}))\leq 0$ on $\R^d$. 

By~\eqref{eq:linear_growth} in Assumption~\aref{ass:Lip}, we have $|\langle b(x),u\rangle|\leq \mu |\langle x,u\rangle|$ for all $x,u\in\R^d$ with $|u|=1$. 
Since $y_1,\ldots,y_k$ are orthonormal vectors, it follows that 
\begin{align*}
\langle b(x),\nabla H_\mathcal{Y}(x)\rangle &= -H_\mathcal{Y}(x)^3 \left\langle \sum_{j=1}^k y_j \langle y_j,x\rangle,b(x)\right\rangle = H_\mathcal{Y}(x)^3\sum_{j=1}^k\langle y_j,x\rangle \langle -y_j,b(x)\rangle \\
& \leq H_\mathcal{Y}(x)^3\sum_{j=1}^k |\langle y_j,x\rangle|\cdot  |\langle -y_j,b(x)\rangle|\leq   H_\mathcal{Y}(x)^3\sum_{j=1}^k |\langle y_j,x\rangle| \mu |\langle -y_j,x\rangle| \\
&=  \mu H_\mathcal{Y}(x)^3\sum_{j=1}^k |\langle y_j,x\rangle|^2=  \mu H_\mathcal{Y}(x)^3|G_\mathcal{Y}(x)|^2\leq \mu H_\mathcal{Y}(x)\quad\text{for all $x\in\R^d$.}
\end{align*}   This implies that the inequality $
\cA H_\mathcal{Y} = \langle b,\nabla H_\mathcal{Y}\rangle + \Tr(a\Hessian(G_\mathcal{Y}))\leq \mu H_\mathcal{Y}$ holds on $\R^d$.
\end{proof}

Let Assumption~\aref{ass:mode}  hold for some $R,\eps,\delta\in(0,\infty)$ and fix $\rho_0\in\cM_{1,R,\eps}(\R^d)$ with its furthest mode $x_0\in\R^d\setminus\{0\}$. 
Theorem~\ref{thm:error} requires both~\aref{ass:mode} and Assumption~\aref{ass:Lip} (with the same $\eps$ and some $\mu\in(0,\infty)$), applied with the principal direction in~\eqref{eq:sigma_bound} given by the furthest mode $y_1\coloneqq x_0/|x_0|$. 
Recall that  by~\ref{ass:mode} we have $|x_0|=R(1+\delta)$, implying  the inequality $H_\mathcal{Y}\leq 1/R$ on $B(x_0,R\delta)$, where $\mathcal{Y}$ is the vector subspace in $\R^d$, spanned by the orthonormal vectors $y_1,\ldots,y_k$ in~\aref{ass:Lip}, and the function  $H_\mathcal{Y}$, used in Proposition~\ref{prop:generator_bound} above, is defined in~\eqref{eq:sigma_bound}. With all this in mind, we proceed with the proof of our main theorem.

\begin{proof}[Proof of Theorem~\ref{thm:error}]
\underline{Lower bounds.}
For $\eps>0$ in~\aref{ass:Lip}, let 
$r_k\in(1,\infty)$ be such that $\pi(\{H_\mathcal{Y}\geq 1/r_k\})>1-\eps/2$, where $H_\mathcal{Y}$ was defined in~\eqref{eq:sigma_bound}.
Denote by $\cA$ the generator in~\eqref{eq:generator} of the forward diffusion $Y$.  By Proposition~\ref{prop:generator_bound} we have $\cA H_\mathcal{Y}\leq \xi \circ H_\mathcal{Y}$ on $\R^d$ for the function $\xi(r) \coloneqq \mu r$, $r\in\RP$. 
The related function $\Xi$ in Lemma~\ref{lem:markov_bound} takes the form $\Xi(u,v)=\log(v/u)/\mu$ for $0<u\leq v<\infty$, implying 
$\Xi(1,\infty) = \infty$. Moreover, for any $T\in [0,\infty)$,
we have $\gamma(u,T) = \exp(\mu T)u$ for $u\in(0,1]$ and $\gamma(0,T)=0<1\leq \gamma(1,T)$. Since $0<1/r_k<1$ we conclude   $C_{r_k,T}=\eta_T(1/r_k)=\exp(-\mu T)/r_k$. Thus by the inequality in~\eqref{eq:main_bound} of Lemma~\ref{lem:markov_bound}, for every $T\in(0,\infty)$
we get
\begin{align}
\label{eq:lower_inequality_DATA_FP}
\|\P_{\rho_0}(Y_T\in \cdot)-\pi\|_{\TV}\geq  \pi(H_\mathcal{Y}\geq &1/r_k)-\rho_0(H_\mathcal{Y}\geq C_{r_k,T})\\
& -\int_{\{H_\mathcal{Y}< C_{r_k,T}\}} r_k\gamma(H_\mathcal{Y}(x),T)\rho_0(\ud x).\nonumber
\end{align}

By~\aref{ass:mode} we have $\rho_0(B(x_0,R\delta)) =b_\rho$ with $x_0\in\R^d\setminus\{0\}$ satisfying $|x_0|=R(1+\delta)$.
Fix $T_c = \frac{1}{\mu}\log \frac{R}{2r_k}>0$ since $r_k<R/2$ by assumption in Theorem~\ref{thm:error}. For any $x\in B(x_0,R\delta)$ we have $H_\mathcal{Y}(x) \leq (1+R^2)^{-1/2}\leq R^{-1}< 2/R=C_{r_k,T_c}$, implying 
$$r_k\gamma(H_\mathcal{Y}(x),T_c)=r_k\exp(\mu T_c)H_\mathcal{Y}(x)   \leq 1/2\qquad\text{and}\qquad
B(x_0,R\delta)\subset \{H_\mathcal{Y}< C_{r_k,T_c}\}.$$
Since $r_k\gamma(H_\mathcal{Y}(x),T_c)\leq 1$ for all $x\in\{H_\mathcal{Y}\leq C_{r_k,T_c}\}$,
the following inequality holds:
\begin{equation}
\label{eq:bound_on_r_k_integral}
\int_{\{H_\mathcal{Y}< C_{r_k,T_c}\}} r_k\gamma(H_\mathcal{Y}(x),T_c)\rho_0(\ud x)\leq \rho_0(\{H_\mathcal{Y}< C_{r_k,T_c}\} \setminus B(x_0,R\delta))+\rho_0(B(x_0,R\delta))/2.
\end{equation}
The inequalities in~\eqref{eq:lower_inequality_DATA_FP} and~\eqref{eq:bound_on_r_k_integral} imply
\begin{align*}
\|\P_{\rho_0}(Y_T\in \cdot)-\pi\|_{\TV}& \geq \pi(H_\mathcal{Y} \geq 1/r_k)-\rho_0(\R^d\setminus B(x_0,R\delta))-\rho_0(B(x_0,R\delta))/2 \\
&\geq (1-\eps/2) -(1-b_{\rho})-b_\rho/2 = (b_\rho-\eps)/2,
\end{align*}
where the inequality $\pi(H_\mathcal{Y}\geq 1/r_{r_k})\geq 1-\eps/2$ holds by the definition of $r_k$ in~\ref{ass:Lip}. This  concludes the proof of the lower bound in Theorem~\ref{thm:error}.

\noindent \underline{Upper bounds.} For any two probability distributions $Q_1,Q_2\in\cM_1(\R^d)$ with respective Lebesgue densities $q_1,q_2$, Pinsker's inequality (see e.g.~\cite[Lemma~2.5(i)]{Tsybakov}) yields $$\|Q_1-Q_2\|_{\TV}\leq \sqrt{\text{KL}(Q_1||Q_2)/2},\quad\text{ where 
$\text{KL}(Q_1||Q_2) \coloneqq \int_{\R^d}q_1(x)\log (q_1(x)/q_2(x))\ud x$}$$ 
is the Kullback-Leibler (KL) divergence between $Q_1$ and $Q_2$.  Pick $\mu_1,\mu_2\in\R^d$ and positive definite matrices $\Sigma_1,\Sigma_2\in\R^{d\times d}$. For $i\in\{1,2\}$, let $Q_i= N(\mu_i,\Sigma_i)$ 
be the Gaussian law on $\R^d$ with mean $\mu_i$ and covariance matrix $\Sigma_i$. Recall the explicit formula for the KL  divergence (see e.g.~\cite[Eq.~(A.23)]{Rasmussen}):
\begin{equation}
\label{eq:KL_normals}
\text{KL}(Q_1||Q_2) = \frac{1}{2}(\Tr(\Sigma_2^{-1}\Sigma_1-\Id) + (\mu_1-\mu_2)^{\intercal}\Sigma_2^{-1}(\mu_1-\mu_2) +\log \mathrm{det}(\Sigma_2\Sigma_1^{-1})).
\end{equation}

Recall that Assumptions~\ref{ass:mode} and~\ref{ass:Lip} hold with some $R,\eps,\mu,\delta\in(0,\infty)$. Denote $B_R\coloneqq B(0,R(1+2\delta))$ and for any  $T\in(0,\infty)$ set $e(T) \coloneqq (1-\exp(-2T\mu))/\mu$. By the triangle inequality for the total variation norm and Pinsker's bound, we obtain
\begin{align}
\nonumber
\|\P_{\rho_0}(X_T\in\cdot)-\pi\|_{\TV} &\leq \|\rho_0(B_R)\P_{\rho_0}(X_T\in\cdot|\{X_0\in B_R\})\\
\nonumber &\quad+\rho_0(\R^d\setminus B_R)\P_{\rho_0}(X_T\in\cdot|\{X_0\in \R^d\setminus B_R\})-\pi\|_{\TV}\\
&\nonumber\leq \rho_0(B_R)\|\P_{\rho_0}(X_T\in\cdot|\{X_0\in B_R\})-\pi\|_{\TV} + \rho_0(\R^d\setminus B_R)\\
\label{eq:TV_KL}&\leq  \rho(B_R)\mathrm{KL}(\P_{\rho_0}(X_T\in\cdot|\{X_0\in B_R\})||\pi)^{1/2}/\sqrt{2} + \rho(\R^d\setminus B_R).
\end{align}

Under $\P_{\rho_0}$, the OU process $X$ is initialised at 
$X_0\sim\rho_0$ and
solves SDE~\eqref{eq:OU}. Thus, given $X_0$, the marginal $X_T$ follows
$X_T|X_0\sim N(X_0/\exp(\mu T),e(T)\Id)$ and 
the invariant measure of $X$ equals $\pi_X = N(0,\Id/\mu)$.
Understanding the KL divergence $\mathrm{KL}(\P_{\rho_0}(X_T\in\cdot|\{X_0\in B_R\})||\pi)$ amounts to 
conditioning on $X_0=y\in B_R$ and applying~\eqref{eq:KL_normals}
to $\mathrm{KL}(N(y/\exp(\mu T), e(T) \Id)||N(0,\Id/\mu))$ for any $y\in B_R$. In particular, we have $\Sigma_1=e(T)\Id$, $\mu_1=y\exp(-2\mu T)$, $\Sigma_2=\Id/\mu$ and $\mu_2=0$, implying
\begin{align*}
\Tr(\Sigma_2^{-1}\Sigma_1-\Id) & = -d \exp(-2\mu T), \qquad
(\mu_1-\mu_2)^{\intercal}\Sigma_2^{-1}(\mu_1-\mu_2)  = \mu\exp(-2\mu T) |y|^2,\\ 
\log \det (\Sigma_2\Sigma_1^{-1}) & = -d \log(1-\exp(-2\mu T)).
\end{align*}
Note that $r+\log(1-r)\geq -r^2$ for all $r\in(0,1/2)$.  Thus,
for $\mu T>\log(2)/2$, we have 
\begin{equation}
\label{eq:tr+log_det_KL}
    \Tr(\Sigma_2^{-1}\Sigma_1-\Id)+\log \det (\Sigma_2\Sigma_1^{-1})=-d \left(\re^{-2\mu T}+\log(1-\re^{-2\mu T})\right)<d\re^{-4\mu T}.
\end{equation}
By the formula in~\eqref{eq:KL_normals} and the inequality in~\eqref{eq:tr+log_det_KL}
we obtain
\begin{align}
\nonumber
\mathrm{KL}(\P_{\rho_0}(X_T\in\cdot|\{X_0\in B_R\})||\pi) &= 
\int_{B_R}\mathrm{KL}(N(y/\exp(\mu T), e(T) \Id)||N(0,\Id/\mu))\rho_0(\ud y)/\rho_0(B_R)\\
\nonumber
& \leq \frac{\mu}{2} \exp(-2\mu T)\int_{B_R}|y|^2 \rho_0(\ud y)/\rho_0(B_R)+d\exp(-4\mu T)/2\\
\label{eq:final_bound_on_KL}
& \leq 
\frac{\mu}{2}\exp(-2\mu T) R^2(1+2\delta)^2+d\exp(-4\mu T)/2,
\end{align} 
where the last inequality holds since  $B_R=B(0,R(1+2\delta))$ is the disc in $\R^d$ with radius $R(1+2\delta)$ and centered at the origin.
Choosing the time $T$ in inequality~\eqref{eq:final_bound_on_KL} equal to $$T_{\mathrm{OU}} \coloneqq \frac{1}{\mu}\max\{\log (2d^{1/4}/\eps^{1/2}),\log (2R(1+2\delta)\sqrt{\mu}/\eps)\}$$ yields $\mathrm{KL}(\P_{\rho_0}(X_{T_{\mathrm{OU}}}\in\cdot|\{X_0\in B_R\})||\pi)\leq \eps^2/2$ (note that $T_{\mathrm{OU}}> (\log 2)/(2\mu)$, implying that the inequalities in~\eqref{eq:tr+log_det_KL} and~\eqref{eq:final_bound_on_KL} hold for $T_{\mathrm{OU}}$). Combining the bound on the KL divergence with~\eqref{eq:TV_KL} yields
$$
\|\P_{\rho_0}(X_{T_{\mathrm{OU}}}\in\cdot)-\pi\|_{\TV}\leq \rho_0(B_R) \eps/2 + \rho_0(\R^d\setminus B_R)< \eps,
$$
where the last inequality follows, since  $\rho_0(\R^d\setminus B_R)<\eps/2$ holds by Assumption~\aref{ass:mode}.
\end{proof}

\begin{proof}[Proof of Corollary~\ref{cor:main}]  By Assumption~\aref{ass:Add}, we have $2r_k\leq R^\beta$. Thus, we obtain $T_c= \frac{1}{\mu}\log(R/(2r_k))\geq (1-\beta)/\mu\log R$ in the Corollary~\ref{cor:main}. Moreover, by Assumption~\aref{ass:Add}, we have $2\sqrt{\mu}(1+2\delta)/\eps\leq R^\beta$ and $R\geq (\eps/\mu)^{1/2}d^{1/4}$, which yields 
$$
\frac{1}{\mu}\max\{\log (2d^{1/4}/\eps^{1/2}),\log (R2(1+2\delta)\sqrt{\mu}/\eps)\} = \frac{1}{\mu}\log (R2(1+2\delta)\sqrt{\mu}/\eps)  \leq \frac{1+\beta} {\mu} \log R. \qquad\text{\qedhere}
$$
\end{proof}

\begin{proof}[Proof of Proposition~\ref{cor:cutoff}]
The inequality $\|\P_{\rho_0}(X_{T_\textrm{OU}}\in\cdot)-\pi_X\|_{\TV}<\eps$ 
 follows from the second inequality in Theorem~\ref{thm:error} for $\mu=1$ with  $T_{\mathrm{OU}}=\log (2R(1+2\delta)/\eps)$, since  in Proposition~\ref{cor:cutoff} we assume in addition $R\geq \eps^{1/2} d^{1/4}$.

The invariant measure $\pi_X$ of the OU process $X$ in~\eqref{eq:OU} with $\mu=1$ is the  standard Gaussian probability measure $N(0,\Id)$ on $\R^d$.   Recall from Assumption~\aref{ass:mode} that the location of the furthest mode  the initial data distribution $\rho_0$ is denoted by $x_0\in\R^d\setminus\{0\}$. Hence the push-forward of  $\pi_X$ under the orthogonal projection onto the line spanned by $x_0/|x_0|$ is a standard Gaussian law on $\R$ with mean zero and variance one.
Fix $c_0 \coloneqq \max\{\sqrt{2\log(1/ \eps)},1\}$ and note that
the bound for the tail of the standard Gaussian yields 
\begin{equation}
\label{eq:normal_inequality}
\pi_X(\{\langle \cdot,x_0/|x_0|\rangle \geq c_0\})\leq \exp(-c_0^2/2)/(c_0\sqrt{2\pi})\leq \eps/2. 
\end{equation} 

For $T\in(0,\infty)$ we have that  $X_T\stackrel{d}{=} X_0\exp(-T)+e(T)Z$, where $e(T) = (1-\exp(-2T))^{1/2}$ and  $Z\sim N(0,\Id)$ is independent of $X_0\sim\rho_0$. Recall that   by Assumption~\ref{ass:mode} we have $\rho_0(B(x_0,R\delta))=b_\rho$ and $|x_0|=R(1+\delta)$. Setting $T_b \coloneqq \log (R/c_0)$, on the event $\{X_0\in B(x_0,R\delta)\}$, we have  
$\langle X_0\exp(-T_b),x_0/|x_0|\rangle \geq \exp(-T_b)R=c_0$  a.s., implying
$$\P_{\rho_0}(\langle X_{T_b},x_0/|x_0|\rangle \geq c_0|\{X_0\in B(x_0,R\delta)\})\geq\P(\langle Z,x_0/|x_0|\rangle\geq0 )=1/2.$$
The first inequality in Proposition~\ref{cor:cutoff} follows from
\begin{align*}
\P_{\rho_0}(\langle X_{T_b},x_0/|x_0|\rangle\geq c_0)&
\geq b_\rho \P_{\rho_0}(\langle X_{T_b},x_0/|x_0|\rangle \geq c_0|\{X_0\in B(x_0,R\delta)\}) \geq b_\rho/2,
\end{align*}
the inequality in~\eqref{eq:normal_inequality} and the bound 
 $$\|\P_{\rho_0}(X_{T_b}\in\cdot)-\pi_X\|_{\TV}\geq \P_{\rho_0}(\langle X_{T_b},x_0/|x_0|\rangle\geq c_0) -\pi_X(\{\langle \cdot,x_0/|x_0|\rangle \geq c_0\}).\qquad\text{\qedhere}$$
\end{proof}

\begin{proof}[Proof of Theorem~\ref{cor:error}]
Theorem~\ref{cor:error} follows from Corollary~\ref{cor:main}
if~\aref{ass:Add} is satisfied for tempered Langevin diffusions with a spherically symmetric stationary measure $\pi$ (i.e. the density of $\pi$ is proportional to $\exp(-h_\pi)$, where $h_\pi:\R^d\to[0,\infty)$ is given by a scalar function $h_\pi(x)=H_\pi(|x|)$). Since~\ref{ass:mode} and conditions~\eqref{eq:mode_to_dim} and~\eqref{eq:pi_concentration}  are assumed in Theorem~\ref{cor:error}, Assumption~\aref{ass:Add} will be satisfied if we show that the tempered Langevin diffusion $Y$, satisfying~\eqref{eq:linear_growth_spheriaclly_symmetric},
also satisfies~\aref{ass:Lip}.

Recall that a tempered Langevin diffusion $Y$ 
satisfying~\eqref{eq:linear_growth_spheriaclly_symmetric}
follows the SDE in~\eqref{eq:tempLan} with coefficients $b$ and $\sigma$ given in~\eqref{eq:temperedLangevinVector}. 
Since $b(x)= -H_\pi(|x|)^{2\ell-1}(H_\pi(|x|)-2\ell)H_\pi'(|x|)x/|x|$, the inequlity in~\eqref{eq:linear_growth_spheriaclly_symmetric} implies~\eqref{eq:linear_growth} in Assumption~\aref{ass:Lip}.
The dispersion coefficient $a=\sigma\sigma^\intercal$ in~\eqref{eq:temperedLangevinVector} takes the form $a(x) =\sigma(x)\sigma^\intercal(x) = 2 h_\pi(x)^{2\ell} \Id$ for all $x\in\R^d$. For any  $k$-dimensional vector space $\mathcal{Y}$ in $\R^d$,  pick orthonormal basis $y_1,\dots,y_k$ and assume $k\geq3$. Then we have 
\begin{align*}
    \sum_{j=1}^k\langle  a(x)y_j,y_j\rangle & = k2 h_\pi(x)^{2\ell}  \geq 6 h_\pi(x)^{2\ell}|G_\mathcal{Y}(x)|^2/(1+|G_\mathcal{Y}(x)|^2) \\
    &=3\langle a(x) G_\mathcal{Y}(x)/(1+|G_\mathcal{Y}(x)|^{2})^{1/2}, G_\mathcal{Y}(x)/(1+|G_\mathcal{Y}(x)|^{2})^{1/2}\rangle,
\end{align*}
for all $x\in\R^d$, where $G_\mathcal{Y}$ is the orthogonal projection mapping $\R^d$ onto $\mathcal{Y}$,
implying~\eqref{eq:sigma_bound} in~\aref{ass:Lip}.
\end{proof}

\section{Conclusion, open problems and future directions}
\label{sec:conclusion}

This paper provides explicit non-asymptotic bounds on the 
convergence rates of diffusion processes initialised at  multi-modal data distributions, which are typical in Denoising Diffusion Probabilistic Models (DDPMs)~\cite{song2019generative,janati2024divide}. We show that  substituting the Ornstein-Uhlenbeck (OU) process with a diffusion process that includes multiplicative noise does not significantly improve the convergence rate. 
Additionally, our results reveal that increasing the distance of the modes from the origin in the initial distribution 
$\rho_0$ results in a cut-off type behaviour in the convergence of the OU process.

Our results establish rigorously that the convergence time of a DDPM grows as a logarithm of the diameter of the support of the initial data distribution for a broad class of ergodic forward diffusions. Since this growth is  ubiquitous in this class, it naturally leads to considering alternatives to forward diffusion processes discussed in this paper. Three such processes, considered in the literature, are as follows: Schrödinger bridges, uniformly ergodic diffusions with superlinear coefficients and kinetic Langevin diffusions. 

Schrödinger bridges relate two arbitrary distributions in fixed time, thus eliminating the convergence error of the forward process completely (see~\cite{de2021diffusion,albergo2023stochastic} and the references therein for details of this approach). 
Certain diffusions with superlinear coefficients are known to exhibit uniform ergodicity~\cite{DownMeynTweedie1995}. Given an error tolerance for the convergence of the forward process, such diffusions would allow for a fixed time horizon independent of the initial data distribution. Since diffusions with superlinear coefficients are difficult to sample~\cite{Livingston19} in general, this naturally motivates the development of new algorithms for sampling such diffusions with a given nice (e.g. spherically symmetric and log-concave) invariant measure and arbitrary initial distribution.

\subsection{Kinetic Langevin diffusions as forward processes in DDPMs}
In numerous machine learning applications, an alternative to the overdamped Langevin diffusions, such as the OU process, is the kinetic (underdamped) Langevin diffusion~\cite{ma2015complete}.\footnote{We are grateful to the anonymous referee for rasing the relevant issues discussed in this subsection and suggesting to develop the result in Proposition~\ref{prop:Kinetic}.} Within the denoising diffusion probabilistic modelling framework, a kinetic Langevin diffusion with a standard Gaussian invariant measure has  been suggested as a replacement for the Ornstein–Uhlenbeck process~\cite{dockhorn2021score}, see also~\cite[Sec.~2.2]{chen2022sampling}. Its dynamics for the position–velocity process $(X,V)$ in $\R^{d}\times\R^d$, given by the hypoelliptic SDE
\begin{equation}
\label{eq:kinticSDE}
\ud X_t = V_t \ud t \quad \&\quad \ud V_t = -(V_t + X_t)\ud t + \sqrt{2}\ud B_t,
\end{equation}
where $B$ is a standard $d$-dimensional Brownian motion, preserves the standard Gaussian distribution $\pi_X\otimes\pi_V$ on $\R^{d}\times\R^d$, where $\pi_X=\pi_V$ is a standard Gaussian distribution on $\R^d$ with zero mean and covariance $\Id$. Differently put, the invariant distribution $\pi_X\otimes\pi_V$ is a product of $2d$ one-dimensional standard Gaussian distributions.

\begin{prop}
    \label{prop:Kinetic}
Let a data distribution $\rho_0\in\cM_1(\R^d)$ satisfy Assumption~\aref{ass:mode} for some $0<\eps,\delta<<1$, $R\in(0,\infty)$ and mode $x_0\in\R^d$. 
Recall $b_\rho\coloneqq \rho_0(B(x_0,\delta R))>3\eps$.
Denote by $\pi_X\otimes\pi_V$ the invariant distribution of a kinetic Langevin process $(X,V)$ following SDE~\eqref{eq:kinticSDE}, which is standard Gaussian  on $\R^{2d}$. 
Let $r_{3+3}\in(1,\infty)$ satisfy 
\begin{equation}
\label{eq:Gausian_concentration_2d}
\pi_X\otimes\pi_V( B_6(r_{3+3})\times\R^{2d-6})>1-\eps/2,
\end{equation}
where
$B_6(r_{3+3})$ is the closed ball in $\R^6$ with radius $(r_{3+3}^2-1)^{1/2}$.
If
$R\geq 2r_{3+3}$,
then 
\begin{align*}
\|\P_{\rho_0\otimes \pi_V}((X_{T_c},V_{T_c})\in\cdot)-\pi_X\otimes\pi_V\|_{\TV}\geq (b_{\rho}-\eps)/2,&\quad\text{where $T_c\coloneqq \log \left(\frac{R}{2r_{3+3}}\right)$.}
\end{align*}
\end{prop}

Note that the ``$\eps/2$-quantile''  $r_{3+3}$ in~\eqref{eq:Gausian_concentration_2d} of the Gaussian distribution on $\R^{2d}$ is independent of the dimension $d$ and is of order  $\sqrt{\log(1/\eps)}$ (as $\eps\to0$).
Proposition~\ref{prop:Kinetic}, together with the upper bound for the OU dynamics in Proposition~\ref{cor:cutoff}, thus shows that using  a kinetic Langevin diffusion with the Gaussian invariant distribution as a forward process in DDPMs (suggested in~\cite{dockhorn2021score}) does not beat the OU benchmark at the task of forgetting the initial data distribution.

\begin{proof}[Proof of Proposition~\ref{prop:Kinetic}]
Let $\mathcal{Y}$ be a three-dimensional vector subspace in $\R^d$ containing the furthest mode $x_0\in\R^d\setminus\{0\}$ of the data distribution $\rho_0$ satisfy Assumption~\aref{ass:mode}. Defined the function $H_{\mathcal{Y}}:\R^d\times\R^d\to(0,1]$ by
$
H_{\mathcal{Y}}(x,v)\coloneqq(1+|G_{\mathcal{Y}}(x)|^2+|G_{\mathcal{Y}}(v)|^2)^{-1/2}$,
where, as in Assumption~\aref{ass:Lip} above, the linear transformation  $G_{\mathcal{Y}}:\R^d\to\mathcal{Y}$ is the orthogonal projection onto the subspace $\mathcal{Y}$. Note that the assumption in~\eqref{eq:Gausian_concentration_2d} yields $\pi(H_{\mathcal{Y}} \geq r_{3+3})>1-\eps/2$.

Let $(X,V)$ follow the kinetic Langevin SDE in~\eqref{eq:kinticSDE}, initialised at $(X_0,V_0)\sim \rho_0\otimes \pi_V$, where $\pi_V$ is a standard $d$-dimensional Gaussian distribution on $\R^d$. Let $\cA$ be the 
extended generator of $(X,V)$, defined in the first display of Section~\ref{sec:proofs} above. Note that $H_{\mathcal{Y}}$ is in the domain of $\cA$ and 
$
\cA H_{\mathcal{Y}} = \langle v,\nabla_x H_{\mathcal{Y}}\rangle - \langle x+v, \nabla_v H_{\mathcal{Y}}\rangle + \Tr(\Hessian_v(H_{\mathcal{Y}})(x,v))$,
where $\nabla_x$, $\nabla_v$, $\Hessian_v$ denote  the gradients in $x$ and $v$ and the Hessian in $v$, respectively. 
As in the proof of Proposition~\ref{prop:generator_bound}, we find 
$\Tr(\Hessian_v(H_{\mathcal{Y}}))\leq 0$.
Since
$\nabla_x H_{\mathcal{Y}}(x,v) = H_{\mathcal{Y}}(x,v)^{3} G_\mathcal{Y}(x)$, 
$\nabla_v H_{\mathcal{Y}}(x,v) = H_{\mathcal{Y}}(x,v)^3 G_\mathcal{Y}(v)$
and $G_\mathcal{Y}=G_\mathcal{Y}^\tra$,
for $(x,v)\in\R^d\times\R^d$ we obtain
\begin{align*}
(\cA H_{\mathcal{Y}})(x,v) &\leq -H_{\mathcal{Y}}(x,v)^3(\langle v,G_\mathcal{Y}(x)\rangle - \langle x+v, G_\mathcal{Y}(v)\rangle) \\
&= H_{\mathcal{Y}}(x,v)^3(-\langle v,G_\mathcal{Y}(x)\rangle +\langle x, G_\mathcal{Y}(v)\rangle + \langle v, G_\mathcal{Y}(v)\rangle)\\
&= H_{\mathcal{Y}}(x,v)^3\langle v, G_\mathcal{Y}(v)\rangle = H_{\mathcal{Y}}(x,v)^3|G_\mathcal{Y}(v)|^2 \leq H_{\mathcal{Y}}(x,v).
\end{align*}

Since $\cA H_{\mathcal{Y}}\leq H_{\mathcal{Y}}$, Lemma~\ref{lem:markov_bound}
applied to the process $\kappa=(X,V)$, for any $T\in(0,\infty)$  and $0<1/r<\exp(T)$ we get
\begin{align*}
\|\P_{\rho_0\otimes\pi_V}((X_T,V_T)\in \cdot)-\pi_X\otimes\pi_V\|_{\TV}&\geq  \pi_X\otimes\pi_V(H_\mathcal{Y}\geq 1/r)-\rho_0\otimes\pi_V(H_\mathcal{Y}\geq C_{r,T})\\
& -\int_{\{H_\mathcal{Y}< C_{r,T}\}} r\gamma(H_\mathcal{Y}(x,v),T)\rho_0\otimes\pi_V(\ud x\otimes\ud v),
\end{align*}
where
$C_{r,T}\coloneqq \exp(-T)/r$ and $\gamma(u,T)=u\exp(T)$ for $u\in(0,1)$.
Applying this inequality with $r=r_{3+3}$ and $T=T_c=\log R-\log(2r_{3+3})$
yields the lower bound in the proposition via an argument completely analogous to the one in the final paragraph of the lower-bound proof of Theorem~\ref{thm:error}.
\end{proof}

Proposition~\ref{prop:Kinetic} establishes a lower bound on the convergence 
of the kinetic Langevin dynamics, started from a data distribution $\rho_0$, 
to the standard Gaussian distribution. 
By Proposition~\ref{cor:cutoff}
this bound matches approximately (for large distance $R$ to the furthest mode of $\rho_0$) that of the Ornstein-Uhlenbeck process.
However, in the context of the applications of kinetic Langevin dynamics as forward processes in DDPMs, several key questions remain open.

\begin{enumerate}
\item \textit{Does cut-off occur as $R\to\infty$?} Establish a matching upper bound and determine whether a cutoff phenomenon arises when the distance $R$ to the furthest mode tends to infinity, as is the case for the OU process, see Proposition~\ref{cor:cutoff} below.

\item \textit{Do perturbed kinetic Langevin diffusions perform better as forward processes in DDPMs?} Perturbed kinetic Langevin diffusions were shown to improve sampling efficiency~\cite{MR3723428,barp2021unifying}. 
Results in the present paper, including Proposition~\ref{prop:Kinetic}, appear to suggest that such  modifications would not significantly speed up convergence in DDPMs. 
A rigorous proof of this claim would require a generalisation of Proposition~\ref{prop:Kinetic} to kinetic Langevin dynamics with an arbitrary friction parameter~\cite{ma2015complete} and  perturbed kinetic Langevin diffusions  considered in~\cite{MR3723428}.

Even though perturbed kinetic Langevin dynamics appear not to exhibit faster convergence than OU processes, they may offer practical benefits when used as the forward diffusion in DDPMs. This makes related convergence characterisation questions for perturbed kinetic Langevin diffusions an interesting future research direction.

\item \textit{Discretisation error of kinetic Langevin diffusions as forward processes in DDPMs.} The smoother paths of underdamped Langevin diffusions often yield smaller discretisation errors than those of overdamped Langevin schemes (see e.g.~\cite{MR4091098} and references therein). A comprehensive rigorous study of how these discretisation errors propagate through the forward dynamics in DDPMs, whether they affect  the learning of the reverse process and how they accumulate overall is essential for a comparison of the DDPMs based on the two forward process choices.

\item \textit{Score matching error (in~\eqref{eq:stable_diff_error}) under kinetic Langevin forward processes.} It has been observed empirically in~\cite{dockhorn2021score} that, in certain scenarios, employing overdamped Langevin  diffusion as a forward process leads to a more efficient computation of the score, a key quantity in DDPMs. Rigorous analysis of  score-matching accuracy across different forward processes, initialised at 
data distributions satisfying Assumption~\aref{ass:mode},
constitutes a crucial open avenue of research.

\end{enumerate}

\subsection{Future directions}

DDPMs have delivered revolutionary advancements across various domains~\cite{song2019generative,song2020score,sohl2015deep,ho2020denoising}. However, while many different approaches exist, there is no clear consensus on the optimal method, see e.g.~\cite{de2021diffusion,conforti2023score,albergo2023stochastic} and the references therein, for discussions of various different choices for denoising processes. A common method to evaluate these diverse approaches is through the analysis of their convergence. This paper provides tools that facilitate such comparisons.  Comparing algorithms by studying their convergence properties has been one of the central themes in applied probability, with extensive literature examining for example various Markov Chain Monte Carlo (MCMC) algorithms~\cite{Andrieu22,brooks1998convergence,Rosenthal04}. To the best of our knowledge, the results in this paper  provide a first step in this direction for DDPMs. In the future, we aim to extend the analysis to the convergence of both forward and reverse diffusion used in DDPMs. Investigating non-asymptotic bounds on the convergence of denoising diffusion algorithms and quantifying a ``cut-off" phenomenon in the convergence of the algorithm are important future research directions. 

We conclude with a remark on our methods and assumptions. The primary role of Assumption~\aref{ass:Add} is to contextualize the results within DDPMs, enabling comparisons between different choices of forward diffusion processes. However, to establish bounds on convergence, we require only milder Assumptions~\aref{ass:mode} and~\aref{ass:Lip}. Moreover, the proof rests on Lemma~\ref{lem:markov_bound} above, which is independent of Assumptions~\aref{ass:mode} and~\aref{ass:Lip}. Our approach may thus remain relevant in frameworks where these assumptions are not applicable.

\appendix
\section{Stability of tempered Langevin diffusions}
\label{app:sec:A}
In order to place the results of this paper in a broader context,  
this section briefly recalls certain key facts concerning the stability of tempered Langevin diffusions with stretched exponential target measures, considered in Section~\ref{subsec:lan} above.

Let $\pi$ be a probability measure on $\R^d$ with  a twice continuously differentiable density. In applications, such as MCMC, it is crucial to construct an ergodic Markov process with $\pi$ as its invariant measure. This is often achieved via a tempered Langevin diffusions, for which, as we shall see, the convergence rate depends on the tail decay of $\pi$. Consider a measure $\pi$ satisfying
\begin{equation}
\label{eq:invariant_tail}
\lim_{r\to\infty} -r^p\log \pi(|\cdot |\geq r) \in(0,\infty) \quad\text{for some $p\in(0,\infty)$}.
\end{equation}

\begin{prop}
\label{prop:lan_conv}
Let $Y$ follow SDE~\eqref{eq:tempLan} with coefficients in~\eqref{eq:temperedLangevinVector}, given by  $\ell\in[0,\infty)$ and an invariant measure $\pi$, satisfying~\eqref{eq:invariant_tail} with $p\in(0,\infty)$. Then the following statements hold.
\begin{enumerate}[label=(\alph*)]
    \item \underline{Case $p\in(0,1)$}. For $\ell \in [0,1/p-1)$, the process $Y$ is subexponentially ergodic: there exist  $c_{-},c_+\in(0,\infty)$, such that  for  $y\in\R^d$ there exist  $C_{-}(y),C_{+}(y)\in(0,\infty)$ satisfying
    $$
    C_{-}(y)\exp(-c_{-} t^{p/(2-p-2\ell p)}) \leq\|\P_y(Y_t\in\cdot)-\pi\|_{\TV}\leq C_{+}(y)\exp(-c_{+} t^{p/(2-p-2\ell p)}),
    $$
    for all $t\in[0,\infty)$.
    For $\ell\in[1/p-1,1/p-1/2]$, the process  $Y$ is exponentially ergodic: there exists $c\in(0,\infty)$, such that for $y\in\R^d$ there exists $C(y)\in(0,\infty)$ satisfying
    $$
    \|\P_y(Y_t\in\cdot)-\pi\|_{\TV}\leq C(y)\exp(-c t)\quad\text{for all $t\in(0,\infty)$.}
    $$
    For $\ell\in(1/p-1/2,\infty)$, the process $Y$ exhibits uniform ergodicity:  there exist constant $c,C\in(0,\infty)$ satisfying
    $$
     \|\P_y(Y_t\in\cdot)-\pi\|_{\TV}\leq C\exp(-c t)\quad\text{for all $t\in[0,\infty)$ and $y\in\R^d.$}
    $$
    \item \underline{Case $p\in[1,2]$}.  The process $Y$ is exponentially ergodic when $\ell\in [0,1/p-1/2]$ and uniformly ergodic otherwise.
    \item \underline{Case $p\in(2,\infty)$}. The process $Y$ is uniformly ergodic for all $\ell\geq 0$.
\end{enumerate}
\end{prop} 

The lower and upper bounds in  Proposition~\ref{prop:lan_conv} follow from~\cite[Thm~3.6]{BM2023} and~\cite[Thm~5.5]{Douc}, respectively. In particular, the lower bounds in~\cite[Thm~3.6]{BM2023} imply that the actual rate of decay of $\|\P_y(Y_t\in\cdot)-\pi\|_{\TV}$ in the stretched exponential case $p\in(0,1)$ depends on the temperature parameter $\ell$.
Recall from the SDE in~\eqref{eq:temperedLangevinVector} and~\eqref{eq:tempLan}
that $\ell=0$ corresponds to the classical Langevin diffusion, while $\ell>0$ yields a tempered Langevin diffusion with unbounded multiplicative noise. Similar results hold for an invariant measure $\pi$ with polynomial tails, corresponding to  $p=0$ in~\eqref{eq:invariant_tail}. However, the coefficients of a tempered Langevin diffusion with $\pi$ that has polynomial tails take the form, different from~\eqref{eq:temperedLangevinVector}; see~\cite[Sec.~3.2]{MR2134115} and~\cite[Sec.~3.1]{BM2023} for more details. As our focus in the present paper is on invariant measures $\pi$ with lighter tails, the details for the polynomial case are omitted.

In this context, tempered Langevin diffusions offer an effective solution: for a given invariant distribution $\pi$, using tempered Langevin diffusions with $\ell>0$ can by Proposition~\ref{prop:lan_conv} significantly improve convergence over the classical Langevin counterpart with $\ell=0$. In particular, when $0<p<<1$, the convergence rate of $\exp(-c t^{p/(2-p)})$ with $\ell=0$ can be enhanced to exponential convergence with the choice of $\ell = 1/p-1$ and uniform ergodicity with $\ell>1/p-1/2$. However, in practical application of ergodic diffusions, such as MCMC, choosing a large temperature parameter $\ell$ has its limitations: for $\pi$, satisfying~\eqref{eq:invariant_tail} with $p\in(0,2]$, choosing $\ell>1/p-1/2$ makes the drift of the tempered Langevin diffusion, given in~\eqref{eq:temperedLangevinVector}, grow superlinearly making sampling numerically unstable~\cite{Livingston19}. Thus enhancing the convergence of sampling algorithms up to exponential rate using tempered Langevin diffusion with $\ell>0$ appears feasible. But achieving uniform ergodicity seems to be out of reach with present sampling methodology.


In this work we explore a related question by comparing the convergence of diffusions with (possibly) different invariant measures that are initialised at a given fixed high-dimensional and multi-modal distribution. We study a broad class of diffusions that can be efficiently simulated. In contrast to Proposition~\ref{prop:lan_conv},  where introducing multiplicative noise accelerated the transport of mass towards the tails, our main result Theorem~\ref{thm:error} implies (among other things) that multiplicative noise does not speed up the transport of mass from the distant modes of the data distribution towards the origin.  Therefore the canonical choice of the OU process is hard to beat in the class of SDEs defined by~\ref{ass:Lip}, which includes tempered Langevin processes. 

\section{Dependence of \texorpdfstring{$r_3$}{r3} on the tails of the noise measure \texorpdfstring{$\pi$}{Pi}}
\label{app:dependence_of_r_3}
Spherically symmetric distributions play a central role as invariant measures of tempered Langevin diffusions in Section~\ref{subsec:lan} above. In particular, we are interested in distributions with densities proportional to  $x\mapsto \exp(-a|x|^p)$, for some $a\in(0,\infty)$ and $p\in(0,2]$, which naturally extend the class of Gaussian measures. The crucial quantity in the context of this paper is the $\eps$-quantile $r_3>1$ of the three-dimensional projection, defined by 
$$
\pi(B_3(r_3)\times \R^{d-3})\geq 1-\eps/2,
$$ where $B(r_3)$ is the closed ball in $\R^3$ with radius $(r_3^2-1)^{1/2}$. Since the quantiles are in terms of three-dimensional projections, we anticipate that $r_3$ will exhibit only a very slow growth as the dimension $d$ increases. The values of quantiles for spherically symmetric distributions on $\R^d$ have one-dimensional integral representations~\cite[Eq.~(1.4)]{Kotz02}). However, except in very special cases such as Gaussian or generalized Laplace distributions on $\R^d$, these expressions are typically intractable analytically. Consequently, we resort to simulations to understand the growth of the quantile $r_3$ with dimension $d$.  Indeed, the numerical experiments  below  strongly suggest that the growth of $r_3$  appears to be logarithmic in $d$ for distributions with densities proportional to $x\mapsto \exp(-a|x|^p)$. 

In the following table, the $\eps$-quantile $r_3$ of the three-dimensional projection of a 
spherically symmetric distribution with density proportional to the function $x\mapsto \exp(-|x|^{p})$, with parameter $p\in\{1,1.2,1.4,1.6,1.8\}$, dimensions $d\in\{3, 30, 300,3000\}$ and $\eps=0.1$, is estimated using $3\cdot 10^5$ independently simulated samples.

\vspace{5mm}

\begin{center}
\begin{tabular}{|c|c|c|c|c|}
\hline
Values of $r_3$ for $\eps=0.1$ 
& $d=3$ & $d=30$ & $d=300$ & $d=3000$ \\ 
\hline
$p=1.8$ & $2.2$ & $2.4$ & $2.8$ & $3.1$ \\ 
\hline
$p=1.6$ & $2.6$ & $3.2$ & $4.3$ & $5.7$  \\ 
\hline
$p=1.4$ & $3.1$ & $4.6$ & $7.5$ & $12.2$ \\ 
\hline
$p=1.2$ & $4.1$ & $7.7$ & $16.1$ & $34.6$\\ 
\hline
$p=1$ & $6.3$ & $15.9$ & $48.7$ & $153.2$  \\ 
\hline
\end{tabular}
\end{center}

\section*{Funding}

MB and AM are supported by EPSRC under grant EP/V009478/1. AM's research is also supported by EPSRC grant EP/W006227/1 and by The Alan Turing
Institute under the EPSRC grant EP/Z532861/1.
The authors would also like to thank the Isaac Newton Institute for Mathematical Sciences, Cambridge,  for support during the INI satellite programmes \textit{Heavy Tails in Machine Learning} and \textit{Diffusions in Machine Learning: foundations, generative models and non-convex optimisation} that took place at The Alan Turing Institute, London, and the INI programme \textit{Stochastic systems for anomalous diffusion},  where work on this paper was undertaken. This work was supported by EPSRC grant EP/R014604/1.

\bibliography{generative}
\bibliographystyle{amsplain}

\end{document}